\documentclass[11pt]{amsart}

\usepackage{amscd,amsthm,amsfonts,amssymb,amsmath}
\usepackage{mathrsfs}
\usepackage[colorlinks=true]{hyperref}
\usepackage{fullpage}
\usepackage[all]{xy}

    \newcommand{\BA}{{\mathbb {A}}} 
    \newcommand{\BC}{{\mathbb {C}}} 
     \newcommand{\BF}{{\mathbb {F}}}

     \newcommand{\BN}{{\mathbb {N}}}
     \newcommand{\BP}{{\mathbb {P}}}
    \newcommand{\BQ}{{\mathbb {Q}}} \newcommand{\BR}{{\mathbb {R}}}

     \newcommand{\BZ}{{\mathbb {Z}}}

     \newcommand{\CH}{{\mathcal {H}}}

    \newcommand{\CO}{{\mathcal {O}}} 
     \newcommand{\CR}{{\mathcal {R}}}

     \newcommand{\fl}{{\mathfrak{l}}}
     
     \newcommand{\fp}{{\mathfrak{p}}}

    \newcommand{\fU}{{\mathfrak{U}}}

    \newcommand{\ab}{{\mathrm{ab}}}

    \newcommand{\Aut}{{\mathrm{Aut}}}

    \newcommand{\disc}{{\mathrm{disc}}}

    \newcommand{\End}{{\mathrm{End}}}

    \newcommand{\Gal}{{\mathrm{Gal}}} \newcommand{\GL}{{\mathrm{GL}}}

    \renewcommand{\Im}{{\mathrm{Im}}}

    \newcommand{\Isom}{{\mathrm{Isom}}}

    \newcommand{\Lie}{{\mathrm{Lie}}}

    \newcommand{\ord}{{\mathrm{ord}}} \newcommand{\rk}{{\mathrm{rank}}}
     \newcommand{\Pic}{\mathrm{Pic}}

    \renewcommand{\mod}{\ \mathrm{mod}\ }

    \newcommand{\SL}{{\mathrm{SL}}}

    \newcommand{\tr}{{\mathrm{tr}}}\newcommand{\tor}{{\mathrm{tor}}}

    \newcommand{\Vol}{{\mathrm{Vol}}}

     \newcommand{\M}{\mathrm{M}}
        \newcommand{\Tr}{\mathrm{Tr}}

\newcommand{\matrixx}[4]{\begin{pmatrix}
#1 & #2 \\ #3 & #4
\end{pmatrix} }        

    \newcommand{\wh}{\widehat}
    
    \newcommand{\pair}[1]{\langle {#1} \rangle}

    \newcommand{\ov}{\overline}

    \newcommand{\lra}{\longrightarrow}
    \newcommand{\ra}{\rightarrow}

    \newcommand{\lrb}[1]{\left(#1\right)}

     \newcommand{\set}[1]{\left\{#1\right\}}                 

    \theoremstyle{plain}
    \newtheorem{thm}{Theorem}[section] \newtheorem{coro}[thm]{Corollary}
    \newtheorem{lem}[thm]{Lemma}  \newtheorem{prop}[thm]{Proposition}
    \newtheorem {conj}[thm]{Conjecture} 
 \newtheorem{proposition}[thm]{Proposition}

\theoremstyle{remark} \newtheorem{remark}{Remark}[section]
\theoremstyle{remark} 
\theoremstyle{remark} 

    \numberwithin{equation}{section}

\begin{document}
\title{On the $8$ case of Sylvester Conjecture}
\author{Hongbo Yin}
\begin{abstract}
Let $p\equiv 8\mod 9$ be a prime. In this paper we give a sufficient condition such that at least one of $p$ and $p^2$ is the sum of two rational cubes. This is the first general result on the  $8$ case of the so-called Sylvester conjecture.
\end{abstract}
\address{School of Mathematics, Shandong University, Jinan, Shandong,  250100,
China}
\email{yhb2004@mail.sdu.edu.cn}
\thanks{The author is supported by NSFC-11701548 and the Young Scholar Program of Shandong University.}
\maketitle

\section{Introduction}
For a nonzero integer $n$, whether $n$ can be written as the sum of two nonzero rational cubes is an old and interesting question which dates back at least to Fermat and has attracted the interest of many mathematicians including Euler and Dirichlet, see \cite{Dickson}. We say $n$ is a \emph{cube sum} if the answer is positive. This problem shares many commons with another famous problem in number theory, the congruent number problem. For example, they both have relations to some twisted families of CM elliptic curves. In fact, our problem is equivalent to ask for a nonzero integer $n$, whether the elliptic curve $E_n: x^3+y^3=n$ has nontrivial rational points. This is a cubic twist family of $E_1$ with CM field $K=\BQ(\sqrt{-3})$.
Without loss of generality, we can assume $n$ is cube-free. Then it is known that $E_n(\BQ)_\tor$ is trivial for $n>2$. In this case, $n$ is a cube sum if and only if $\rk_\BZ E_n(\BQ)>0$. For primes and square of primes, we have the following result and conjecture.

\begin{thm}[Pepin, Lucas, Sylvester]
Assume $p\equiv 2,5\mod 9$ is an odd prime, then both $p$ and $p^2$ are not cube sums.
\end{thm}

\begin{conj}\label{conj}
Assume $p\equiv 4,7,8\mod 9$ is a prime, then both $p$ and $p^2$ are cube sums.
\end{conj}
Conjecture \ref{conj} is usually called Sylvester conjecture but it is in fact indicated by some $3$-descent computation of Selmer \cite{Selmer51} and first formally proposed by Birch and Stephen based on the BSD conjecture \cite{BS}.  More explicitly, for prime $p$, Selmer's result implies
\[\rk_\BZ E_p(\BQ)\leq\begin{cases}0,& p\equiv 2,5\mod 9;\\
                         1,&  p\equiv 4,7,8\mod 9;\\
                         2,&  p\equiv 1\mod 9.\end{cases}\]
and Birch-Stephen's sign computation shows
\[\epsilon(E_p)=\begin{cases}
                         -1,&  p\equiv 4,7,8\mod 9;\\
                         +1,&  \text{otherwise},\end{cases}\]
where $\epsilon(E_p)$ is the sign in the functional equation of the Hasse-Weil $L$-function $L(s,E_p)$.
Then the BSD conjecture implies Conjecture \ref{conj} for $p$. The same reasoning also works for $p^2$. For more results on the cube sum problem, please refer to \cite{DV1}\cite{HSY19}.

Elkies announced a proof of Conjecture \ref{conj} for all primes $p\equiv 4,7\mod 9$, but he never published any details, see \cite{DV1}\cite{DV17}. The only known results towards Conjecture \ref{conj} is the following theorem of Dasgupta and Voight in 2006 \cite{DV1}\cite{DV17}. 

\begin{thm}[Dasgupta-Voight]\label{47}
Let $p\equiv 4,7\mod 9$ be a prime such that $3\mod p$ is not a cubic residue. Then both
$p$ and $p^2$ are cube sums. 
\end{thm}

However, nearly 70 years have passed and there is no general result about the $8$ case of Conjecture \ref{conj} so far. It is believed that the $8$ case is decidedly much more difficult.
In this paper, we prove the following theorem which is the first general result in this case and also illustrates its complexity. 

\begin{thm}\label{main}
Let $p\equiv 8\mod 9$ be a prime such that the equation 
\[x^9-24x^6+ 3x^3+1-9(\sqrt[3]{3}-1)x^2(x^3+1)^2=0\]
does not have solutions in $\BF_p$, then at least one of $p$ and $p^2$ is the sum of two rational cubes. 
\end{thm}
The method to prove Theorem \ref{main} is traditional by constructing certain Heegner points and then proving their nontriviality. Our proof is largely inspired by Dasgupta and Voight's proof of Theorem \ref{47}. However the situation is very different here. In the $p\equiv 4,7\mod 9$ case, $p$ splits into two primes $\fp$ and $\bar{\fp}$ in the CM field $K=\BQ(\sqrt{-3})$. Dasgupta and Voight distinguished the Heegner points from the torsion points by comparing their coordinates modulo the different primes $\fp$ and $\bar{\fp}$. In the $p\equiv 8\mod 9$ case, $p$ is inert in $K$ and the single coordinates modulo $p$ is not enough to characterize the nontrivility of the Heegner points. However, for the cubic twist family elliptic curves $E_n$, $3$ is another special prime we can use. The elliptic curve $E_n$ has bad reduction at $3$ which is a terrible feature usually, but a good feature in our case. Moreover, It is quite surprising that Gross and Zagier's theory of singular moduli plays a key role in the investigation at prime $3$.
At last, we can prove the nontriviality of the Heegner points by combining both the information modulo $p$ and $3$.

Although Theorem \ref{main} cannot tell which of $p$ and $p^2$ is a cube sum, we will give two explicit Gross-Zagier formulae relating
the central derivatives of $L$-functions and the heights of our Heegner points in Theorem \ref{thm:GZ} below. This can conjecturally tell us whose rational points of $E_p$ and $E_{p^2}$ can be generated by our Heegner points. For more details, see section \ref{app}.

The paper is organized as follows. In section 2, we provide the results about field extensions we need. In section 3, we construct our Heegner points and study the Galois action on them. In section 4, we prove a congruence relation of singular modulus. In section 5, we study the reduction of the Heegner points modulo the primes above $3$. In section 6, we prove our main theorem. In section 7, we give the explicit Gross-Zagier formulae. In section 8, we compute the example of the prime 17 to illustrate our results.

\subsection*{Acknowledgments} The author would like to thank John Voight for many useful communications. He also wants to thank Benedict Gross, Don Zagier, Yingkun Li, Tonghai Yang, Bianca Viray for many useful discussions on singular moduli. He also thanks Jianing Li for the discussions on the density of primes satisfying our conditions. One proof of a key identity in this paper is communicated to the author by `GH from MO' through mathoverflow, thank his help and the good website. Many of the computations in the research are done on the Sagemath system \cite{Sage}, thanks are also given to the Sagemath team. Finally, thank referee for carefully reading the original paper and many valuable suggestions.

Part of this work is finished during the author's one year stay (2019-2020) in Max-Planck Institute for Mathematics, Bonn. He is grateful to its hospitality and financial support.

\section{Ring class fields }
For any integer $c\geq 1$, let $\CO_c$ be the order of $K$ of conductor $c$ and let $H_c$ be the ring class field of conductor $c$, i.e., $\Gal(H_c/K)\simeq\wh{K}^\times/ K^\times\wh{\CO_c}^\times$ where $\wh{K}=K\otimes \wh{\BZ}$ is the finite adele of $K$ and $\wh{\CO_c}=\CO_c\otimes\wh{\BZ}$. Let $\sigma: \wh{K}^\times\ra \Gal(K^\ab/K)$ be the Artin reciprocity law and we denote by $\sigma_t$ the image of $t\in\wh{K}^\times$. In the rest of the paper, $\omega=(-1+\sqrt{-3})/2$ is a third root of unity and $p$ will always be a prime congruent to $8$ modulo $9$.  For an element $\alpha$ of $K$, we will use $\alpha_v$ to denote the embedding of $\alpha$ into $\wh{K}^\times$ with the $v$-place $\alpha$ and all other places $1$ (but we will use $\alpha_3$ instead of $\alpha_{\sqrt{-3}}$ below for simplicity, e.g., we write $\omega_3$ for $\omega_{\sqrt{-3}}$).
We have the following field extension results. 

\begin{prop}\label{LCF}Let $p\equiv 8\mod 9$ be a prime.
\begin{itemize}
\item[1.] The field $H_{9p}=H_{3p}(\sqrt[3]{3})$ has Galois group 
$$\Gal(H_{9p}/H_{3p})\simeq \langle \sigma_{1+3\omega_3}\rangle\simeq \BZ/3\BZ,$$
Moreover,
\[\left(\sqrt[3]{3}\right)^{\sigma_{1+3\omega_3}-1}=\omega^2.\]
\item[2.] $K(\sqrt[3]{p})\subset H_p\subset H_{3p}$ and 
$$\Gal(H_{3p}/H_p)=\langle\sigma_{\omega_3}\rangle\simeq\BZ/3\BZ,\ \ \ \Gal(H_p/K)=\pair{\sigma_{x_p}}\simeq \BZ/((p+1)/3)\BZ,$$
where $x_p$ is the generator of $\CO_{K,p}^\times/\BZ_p^\times(1+p\CO_{K,p})\CO_K^\times$. Hence $p$ is totally ramified in $H_p/K$.
\item[3.] 
Let $\sigma_{\sqrt{-3}}\in\Gal(H_p/K)$ be the Frobenius corresponding to $\sqrt{-3}$ and $F=H_p^{\sigma_{\sqrt{-3}}}$, then $[H_p:F]=2$ and $\sqrt{-3}$ is totally split in $F/K$ and inert in $H_p/F$.
\item[4.] Assume $3^k\mid p^2-1$, then $\zeta_{3^k}\in \CO_{K,p}^\times$ where $\zeta_{3^k}$ is a primitive $3^k$-th root of unity. 
\end{itemize}
\end{prop}
\begin{proof}
We prove the results in 1,2,3,4 
successively and separately.
\begin{enumerate}
\item The Galois group
\[\Gal(H_{9p}/H_{3p})\simeq K^\times \wh{\CO_{3p}}^\times/K^\times\wh{\CO_{9p}}^\times\]
is cyclic of order $3$ and generated by $1+3\omega_3$.
The ideal $7\CO_K=(1+3\omega)(1+3\omega^2)$ and let $v$ be the place corresponding to the prime ideal $(1+3\omega)$. Then by the local-global principle, we have
\[\left(\sqrt[3]{3}\right)^{\sigma_{1+3\omega_3}-1}=(1+3\omega_3, 3)_3=(1+3\omega_v, 3)_v^{-1}=3^{-2}\mod (1+3\omega)=\omega^2,\]
where $(\ ,\ )_w$ denotes the $3$-rd Hilbert symbol over $K_w$. This finishes the proof of assertions in 1.

\item The Galois group
\[\Gal(H_{3p}/H_{p})\simeq K^\times \wh{\CO_{p}}^\times/K^\times\wh{\CO_{3p}}^\times\]
is cyclic of order $3$ and generated by $\omega_3$. And
\[\Gal(H_{p}/K)\simeq \wh{K}^\times/ K^\times\wh{\CO_{p}}^\times\simeq \CO_{K,p}^\times/\BZ_p^\times(1+p\CO_{K,p})\CO_K^\times\]
is cyclic of order $(p+1)/3$. Since $K(\sqrt[3]{p})$ is dihedral over $\BQ$, by \cite[Theorem 9.18]{Cox89}, $K(\sqrt[3]{p})$ is contained in a ring class field of $K$. To prove that $K(\sqrt[3]{p})\subset H_p$, it is enough to prove that 
$$(1+p\CO_{K,p})\prod_{v\nmid p}\CO_{K,v}^\times$$ 
fixes $\sqrt[3]{p}$ under the Artin map. Since $K(\sqrt[3]{p})/K$ is unramified outside $3p$, $\prod_{v\nmid 3p}\CO_{K,v}^\times$  fixes $\sqrt[3]{p}$. Using the Hilbert symbol, it is clear that $(1+p\CO_{K,p})$ fixes $\sqrt[3]{p}$. Finally, we look at the $3$-adic place. Let $\CO_{K,3}$ be the completion of $\CO_K$ at the unique place above $3$. Since $1+9\CO_{K,3}\subset (K^\times_3)^3$ and
\[ \CO_{K,3}^\times/\BZ_3^\times(1+9\CO_{K,3})=\langle \omega_3\rangle^{\BZ/3\BZ}\times\langle1+3\omega_3\rangle^{\BZ/3\BZ},\] 
it suffices to prove $\omega_3$ and $1+3\omega_3$ fix $\sqrt[3]{p}$.
But $p\equiv 8\mod 9$, we have
\[(\sqrt[3]{p})^{\sigma_{\omega_3}-1}=(\omega_3,p)_3=(\omega_3,8)_3=1.\]
The same is true for $\sigma_{1+3\omega_3}$. The last sentence in 2 follows from the fact that $\Gal(H_p/K)$ is generated by units. This finishes the proof of assertions in 2.

\item Let $\mu$ be a place of $H_{p}$ above $3$. Since $H_p/K$ is only ramified at $p$, we know that $\sqrt{-3}$ is unramified and 
$\Gal(H_{p,\mu}/K_3)\simeq \Gal(k_\mu/\BF_{3})$ is generated by the Frobenius $\sigma_{\sqrt{-3}}$ under the local Artin map.
We have the following commutative diagram
\[\xymatrix{\Gal(H_{p,\mu}/K_3)\ar@{^{(}->}[d]\ar[r]^{\simeq}& K_3^\times / \mathrm{Norm} (H_{p,\mu}^\times)\ar@{^{(}->}[d]\\\Gal(H_{p}/K)\ar[r]^{\simeq\quad }& \wh{K}^\times / K^\times \wh{\CO_{p}}^\times.}\]
But $\sqrt{-3}$ is of order $2$ in $\wh{K}^\times / K^\times \wh{\CO_{p}}^\times\simeq  \CO_{K,p}^\times/\CO_{K}^\times\BZ_p^\times(1+p\CO_{K,p})$. We see $[H_{p,\mu}:K_3]=2$ and the assertions in 3 of the proposition is clear.

\item The order of $\BF_{p^2}^\times$ is divided by $3^k$. Then $x^{3^k}-1$ splits completely in $\BF_{p^2}$. By the Hensel Lemma, $x^{3^k}-1$ splits completely in $K_p$.  The action of $\sigma_{\zeta_{3^k,p}}$ on $\sqrt[3]{p}$ is by the formula of Hilbert symbol.
\end{enumerate}
\end{proof}

In summary, we have the following extension diagram.

\[\xymatrix{&H_{9p}=H_{3p}(\sqrt[3]{3})\ar@{-}[dl]^{3}_{\pair{\sigma_{1+3\omega_3}}}\ar@{-}[d]\ar@{-}[dr]&\\
            H_{3p}\ar@{-}[d]_{3}^{\pair{\sigma_{\omega_3}}}&H_{p}(\sqrt[3]{3})\ar@{-}[dd]\ar@{-}[dl]&H_9\ar@{=}[ddd]\\
            H_{p}\ar@{-}[dd]_{(p+1)/9}&&\\
            &L_{(3,p)}=K(\sqrt[3]{3},\sqrt[3]{p})\ar@{-}[dr]^{3}\ar@{-}[d]^{3}\ar@{-}[dl]_{3}&\\
            L_{(p)}=K(\sqrt[3]{p})\ar@{-}[dr]^{3}&L_{(9p)}=K(\sqrt[3]{9p})\ar@{-}[d]^{3}&L_{(3)}=K(\sqrt[3]{3})\ar@{-}[dl]_{3}\\
            &K\ar@{-}[d]^{2}&\\
            &\BQ&\\
            }\]

Finally, we give a lemma which will be used later.

\begin{lem}\label{prime}
Let $p\equiv 2$ mod $3$ be a prime, then $x^3\equiv 3\mod p$ always has one solution.
\end{lem}
\begin{proof}
By Fermat's little theorem, $3\equiv 3^{2-p}=(3^{\frac{2-p}{3}})^3 \mod p$, so $x^3\equiv 3\mod p$ always has solutions. If $x_1$ and $x_2$ are both solutions, then $\lrb{\frac{x_1}{x_2}}^3\equiv 1\mod p$. But $\BF^\times_{p}$ is of order $p-1$ prime to $3$, so $\frac{x_1}{x_2}=1$, i.e., $x_1=x_2$.
\end{proof}


\section{Construction of Heegner points}
For convenience, we will use the adelic formulation in \cite{Tian}\cite{CST17}. For $X$ an algebraic curve defined over $\BQ$ and  $F$  a field extension of $\BQ$, we denote by $\Aut_F(X)$ the group of algebraic automorphisms of $X$ which are defined over $F$.
Let $$\CH=\{z\in \BC|\, \Im(z)>0\}$$ be the Poincare upper half plane and usually we will use $\tau$ to denote the points on $\CH$. The group $\GL_2(\BQ)^+$ acts on $\CH$ by linear fractional transformations.

Let $U_0(243)$ be the open compact subgroup of $\GL_2(\widehat{\BZ})$ consisting of matrices
$\left (
\begin{array}{cc}
a&b\\c&d
\end{array}
\right )$ such that $c\equiv 0\mod 243$, and let
$\Gamma_0(243)=\GL_2(\BQ)^+\cap U_0(243)$. Let $X_0(243)$ be the
modular curve over $\BQ$ of level $\Gamma_0(243)$ whose underlying Riemann surface is
\[X_0(243)(\BC)=\GL_2(\BQ)^+\backslash\left (\CH\bigsqcup\BP^1(\BQ)\right )\times \GL_2(\BA_f)/U_0(243)\simeq \Gamma_0(243)\backslash \CH\bigsqcup \Gamma_0(243)\backslash \BP^1(\BQ).\]
It has genus $19$. Define $N$ to be the normalizer of $\Gamma_0(243)$ in $\GL_2^+(\BQ)$.
It follows from \cite[Theorem 1]{KM1988} that the linear fractional action of $N$ on $X_0(243)$ induces an isomorphism
\[N/\BQ^\times\Gamma_0(243)\simeq \Aut_{\ov{\BQ}}(X_0(243)).\]
Moreover, all the algebraic automorphisms in $\Aut_{\ov{\BQ}}(X_0(243))$ are defined over $K$.
By \cite[Theorem 8]{AL1970}, the quotient group $N/\BQ^\times\Gamma_0(243)\simeq S_3\rtimes \BZ/3\BZ$, where $S_3$ denotes the symmetric group with $3$ letters which is generated by the Atkin-Lehner operator $W=\begin{pmatrix}0&1\\-243&0\end{pmatrix}$ and the matrix $A=\begin{pmatrix}28&1/3\\ 81&1\end{pmatrix}$, and the subgroup $\BZ/3\BZ$ is generated by the matrix $B=\begin{pmatrix}1&0\\81&1\end{pmatrix}$. Let $C=\begin{pmatrix}1&1/9\\-27&-2\end{pmatrix}$, then $C$ normalizes $\pair{\Gamma_0(243),A}$.

Put
\[U=\langle U_0(243),C,A\rangle\subset \GL_2(\BA_f).\]
Then $\BQ^\times\backslash \BQ^\times U$ is an open compact subgroup of $\BQ^\times\backslash \GL_2(\BA_f)$.
Put $$\Gamma =\GL_2(\BQ)^+\cap U=\langle \Gamma_0(243),C,A\rangle,$$ and let $X_\Gamma$ be the
modular curve over $\BQ$ of level $\Gamma$ whose underlying Riemann surface is
\[X_\Gamma(\BC)=\GL_2(\BQ)^+\backslash\left (\CH\bigsqcup\BP^1(\BQ)\right )\times \GL_2(\BA_f)/U\simeq \Gamma \backslash \CH\bigsqcup \Gamma\backslash \BP^1(\BQ).\]
It was communicated to the author by John Voight that the space of holomorphic modular symbols for $\Gamma_0(243)$ is of dimension $28$ and the fixed subspace of $A$ and $C$ is two dimensional. One of the forms is an Eisenstein series and the other is the form for the isogeny class of elliptic curves 243b in Cremona Label. After integrating this, we get the optimal curve 243b1. So $X_\Gamma$ is a smooth projective curve over $\BQ$ isomorphic to the elliptic curve 243b1 in Cremona Label over $\BC$. Moreover, $X_\Gamma$ has two cusps
\[\Gamma\backslash \BP^1(\BQ)=\{[\infty],[0]\},\] which are all rational over $\BQ$. This gives $X_\Gamma$ the structure of an elliptic curve over $\BQ$. 
Let $N_\Gamma$ be the normalizer of $\Gamma$ in $\GL_2(\BQ)^+$. We have a natural embedding
\begin{equation}\label{Phi}\Phi: N_\Gamma/ \BQ^\times\Gamma\hookrightarrow \Aut_{\ov{\BQ}}(X_\Gamma)\simeq \CO_K^\times\ltimes X_\Gamma(\ov{\BQ}),\end{equation}
where $\CO_K^\times $ embeds into $\Aut_{\ov{\BQ}}(X_\Gamma)$ by complex multiplications and $X_\Gamma(\ov{\BQ})$ embeds into $\Aut_{\ov{\BQ}}(X_\Gamma)$ by translations.

The matrices
\[B=\begin{pmatrix}1&0\\81&1\end{pmatrix}\quad \text{and}\quad W=\begin{pmatrix}0&1\\-243&0\end{pmatrix}\]
lie in $N_\Gamma$, and hence induce automorphisms $\Phi(B)$ and $\Phi(W)$ of $X_\Gamma$. The matrix $W$ also normalizes $U_0(243)$, so for any $P=[\tau,g]\in X_\Gamma$ with $g\in \GL_2(\BA_f)$, $\Phi(W)(P)=[\tau, gW^{-1}]$ which is defined over $\BQ$ by functoriality of the canonical models of Shimura variety. However $B$ does not normalize $U_0(243)$. But for $P=[\tau,\gamma]\in X_\Gamma$ with $\gamma\in \GL_2(\BQ)^+$, the action of $\Phi(B)$ can still be written as $\Phi(B)(P)=[\tau,\gamma B^{-1}]$. But this time we cannot conclude it is defined over $\BQ$. In fact, it is defined over $K$. 

For any $n$, $E_n$ has the Weierstrass equation $y^2=x^3-432n^2$ and is 3-isogenous to $E^n:y^2=x^3+16n^2$, see \cite[P123]{Stephens}. In particular, $E_3$ is 3-isogenous to the optimal elliptic curve $E^3$ which is of conductor $243$ and modular degree $9$, i.e., Cremona Label 243b1. From this we can see that $X_\Gamma\cong E^3$.
We have the explicit modular parametrization
\begin{equation}\label{modular}
\psi:X_0(243)\lra E^3, \ \ \ \ \tau \mapsto \left(x(\tau), y(\tau)=-\frac{8\eta(9\tau)^4}{\eta(27\tau)^4}-12\right).
\end{equation}
To see this, the eta-quotient in the formula is a modular function on $X_0(243)$ by the Ligozat criterion \cite{Ligozat} and invariant under the action of $A$ and $C$ by the transformation properties of the eta functions.
Then we can get the result by comparing the $q$-expansion of $y(\tau)$ with the modular parametrization up to the Hecke bound using Sagemath.
\begin{prop}\label{action}
\begin{itemize}
\item[(1).] Through the modular parametrization $\psi$,  $X_\Gamma$ can be identified with $E^3$ such that $[\infty]=O$ and the cusp $[0]$ has coordinates $(0,-12)$.
\item[(2).] The embedding in (\ref{Phi}) in fact induces the following embedding
\[\Phi:N_\Gamma/ \BQ^\times\Gamma\hookrightarrow\CO_K^\times\ltimes \Gamma\backslash \BP^1(\BQ).\]
Moreover,  for any point on $P\in X_\Gamma$, we have
\[\Phi(B)(P)=[\omega^2]P,\quad \Phi(W)(P)=[-1]P+[0].\]
In particular, the automorphisms $\Phi(B)$ is defined over $K$.
\end{itemize}
\end{prop}
\begin{proof}For (1), we just evaluate $y(\tau)$ at the cusps $\infty$ and $0$. 
For (2), for any $M\in N_\Gamma$ and $P\in X_\Gamma$, $\Phi(M)(P)=[\alpha]P+S$, where $\alpha\in \CO_K^\times,S\in X_\Gamma(\ov{\BQ})$. Taking $P=[\infty]$, we see $S=\Phi(M)([\infty])\in \Gamma\backslash \BP^1(\BQ)$. The formula for $B$ is in \cite[Proposition 2.2.7]{DV17}. The formula for $W$ can be derived from its actions on $[\infty]$ and $[0]$ directly.  
\end{proof}

Let $p\equiv 8\mod 9$ be a prime. Let $\rho:K\rightarrow \M_2(\BQ)$ be the normalized embedding with fixed point $\tau_0=p\omega/(9\sqrt{-3})\in \CH$, i.e., we have
\[\rho(t)\begin{pmatrix}\tau_0\\1\end{pmatrix}=t\begin{pmatrix}\tau_0\\1\end{pmatrix},\quad \textrm{for any $t\in K$}.\]
Note that $$\tau_0=M\omega,\quad M=\matrixx{\frac{p}{9}}{0}{2}{1}.$$
Then the embedding $\rho:K\rightarrow\M_2(\BQ)$ is explicitly given by
\begin{equation}\label{emb1}
\rho(\omega)=M\matrixx{-1}{-1}{1}{0}M^{-1}=
\begin{pmatrix}1&-p/9\\27/p&-2\end{pmatrix}
\end{equation}
and
\begin{equation}\label{emb2}
\rho(a+b\omega)=\begin{pmatrix}a&0\\ 0&a\end{pmatrix}+\begin{pmatrix}b&0\\ 0&b\end{pmatrix}\cdot\rho(\omega).
\end{equation}
Recall $\sigma: \wh{K}^\times\ra \Gal(K^\ab/K)$ is the Artin reciprocity law and we denote by $\sigma_t$ the image of $t\in\wh{K}^\times$.   Let $P_0=[\tau_0,1]$ be the CM point on $X_\Gamma$ and $P_{\tau_0}=\psi(P_0)$ its image in $E^3$.

\begin{thm}\label{co:Galois}
The point $P_0\in X_\Gamma(H_{9p})$ satisfies
\[P_0^{\sigma_{1+3\omega_3}}=[\omega]P_0,\quad P_0^{\sigma_{\omega_3}}=P_0.\]
\end{thm}
\begin{proof}
By Shimura's reciprocity law \cite[Theorems 6.31 and  6.38]{Shimurabook}, we have
\[P_0^{\sigma_t}=[\tau_0, t],\quad t\in \wh{K}^\times.\]
Since $\wh{\CO_{9p}}^\times\subset\wh{K}^\times\cap U$ by (\ref{emb2}), we see $P_0$ is fixed by the Galois action of $\wh{\CO_{9p}}^\times$ and thus defined over the ring class field $H_{9p}$, and the Galois actions $\sigma_{1+3\omega_3}$ and $\sigma_{\omega_3}$ are clear from the following computations and Proposition \ref{action}
\[A^2B^2(1+3\omega_3)=\left(\begin{pmatrix}783/p +9508&-2377p/3 -145/3\\ 2268/p+27540& -2295p-140\end{pmatrix}_3, A^2B^2\right)\in U,\]
\[C^2\omega_3=\left(\begin{pmatrix}-3/p-2&2p/9+2/9\\ 27/p+27& -3p-2\end{pmatrix}_3, C^2\right)\in U,\]
where the subscript $3$ denotes the $3$-adic component of the adelic matrices.
\end{proof}
Denote
\[\phi: E^3: y^2=x^3+16\cdot 9\ra E^1: y^2=x^3+16,\ \  (x,y)\mapsto (x/\sqrt[3]{9}, y/3),\]
the isomorphism between $E^3$ and $E^1$ which is defined over $\BQ(\sqrt[3]{3})$.
\begin{coro}\label{deffield}
$P_{\tau_0}\in E^3(H_p(\sqrt[3]{3}))$ with $\phi(P_{\tau_0})\in E^1(H_p)$.
\end{coro}
\begin{proof}The modular parametrization $\psi$ is defined over $\BQ$.
Write $P_{\tau_0}=(x(\tau_0), y(\tau_0))$. By Theorem \ref{co:Galois} and Proposition \ref{LCF}, $x(\tau_0)/\sqrt[3]{9}\in H_p$ and $y(\tau_0)\in H_p$. So $\phi(P_{\tau_0})=(x(\tau_0)/\sqrt[3]{9}, y(\tau_0))\in E^1(H_p)$.
\end{proof}

We define our Heegner point to be
\[z=\tr_{H_p(\sqrt[3]{3})/L_{(3,p)}}P_{\tau_0}\in E^3(L_{(3,p)}).\]
Then $\phi(z)\in E^1(L_{(p)})$. Here $L_{(3,p)}=K(\sqrt[3]{3},\sqrt[3]{p})$ and $L_{(p)}=K(\sqrt[3]{p})$.

\section{Singular moduli} In this section we prove a key congruence relation which we will need using Gross-Zagier's theory of singular moduli \cite{GZ85} and its generalization by Lauter and Viray \cite{LV}. The readers who are not interested in the details can just look at Theorem  \ref{singularmoduli} and continue reading the next section.

First of all, let us recall the setting in \cite{GZ85} and \cite{LV}. Let $d_1=-3p^2$ and $d_2=-3\cdot 2^2$ be the discriminants of the imaginary quadratic points $27\tau_0$ and $\sqrt{-3}$ respectively. For $i=1,2$, by \cite{ST}(see also \cite[Page 9213]{LV}), there exists a number field $L$ which is only ramified at $p$ and $2$ over $K$, such that, for every prime $q$ and every $\SL_2(\BZ)$ equivalent class $[\tau_i]$ of discriminant $d_i$, there exists an elliptic curve $E(\tau_i)/\CO_{L}$ with good reduction at all primes above $q$ and $j(E(\tau_i))=j(\tau_i)$. Fix a rational prime $\ell$ and a prime $\mu$ of $\CO_L$ above $\ell$. Let $A$ be the ring of integers of $L_{\mu}^{\text{unr}}$ the maximal unramified extension of $L_\mu$. By \cite[Proposition 2.3]{GZ85}, we have
\begin{equation}\label{valuation}
v_{\mu}(j(\tau_1)-j(\tau_2)) =\frac{1}{2}\sum_{n\geq 1}\sharp \Isom_{A/\mu^n}(E(\tau_1),E(\tau_2)).
\end{equation}

\begin{thm}\label{singularmoduli}
$j(27\tau_0)-j(\sqrt{-3})\equiv 0\mod 2^43^55^3$.
\end{thm}

\begin{proof}Note $j(\sqrt{-3})=2^43^35^3$, the result about primes $2$ and $5$ is just \cite[Corollary 2.5]{GZ85}. Now we focus on the prime $3$. Let $\ell=3$ and $\mu$ a prime of $L$ above $3$. Since, the class number of $\CO_{d_2}$ is one, we denote $E'=E(\sqrt{-3})$ in the proof.

For $n,m\in\BZ$ such that $n>0, m\geq 0$ and any elliptic curve $E/A$ with CM by $\CO_{d_1}$ and with good reduction, define (see \cite{LV} for the unexplained definitions)
$$S_n(E/A):=\set{\phi\in\End_{A/\mu^n}(E):\phi^2+3=0,\ \BZ[\phi]\hookrightarrow\End_{A/\mu}(E)\ \text{optimal away from $3$}},$$
$$S_n^{\Lie}(E/A):=\set{\phi\in S_n(E/A): \phi=\sqrt{-3}\ \text{in Lie($E\mod \mu^n$) and Lie($E\mod \mu$)}},$$
$$S_{n,m}(E/A):=\set{\phi\in S_n(E/A):\disc(\CO_{d_1}[\phi])=m^2},$$
$$S_{n,m}^{\Lie}:=S_{n,m}(E/A)\cap S_{n}^{\Lie}(E/A).$$
(Note that in our case $\tilde d_2=d_2$ and $s_2=0$ in \cite[section 3.2]{LV} and we choose $\delta=\sqrt{-3}$ instead of $-3+\sqrt{-3}$). It is known that 
\begin{equation}\label{Sn}
S_{n}(E/A)=\sum_{\substack{m=\frac{d_1d_2-x^2}{4}}}S_{n,m}(E/A),\ \ \ S_{n}^{\Lie}(E/A)=\sum_{m=\frac{d_1d_2-x^2}{4}}S_{n,m}^{\Lie}(E/A),
\end{equation}
since of the form $\frac{d_1d_2-x^2}{4}$ is a necessary condition for $m$ to be a discriminant of the orders, see \cite[Page 9215]{LV}. 
By  \cite[(3.1)]{LV}, we have $\sharp\Isom_{A/\mu^n}(E,E')=2\sharp S_n^{\Lie}(E)$. From the proof of \cite[Proposition 8.1]{LV}, we know that $S_{2,m}(E/A)=S_{1,m}(E/A)=S_{1,m}^{\Lie}(E/A)$ and for $n\geq 2$, $\sharp S_{n+1,m}(E/A)=2\sharp S_{n,m}^{\Lie}(E/A)$ ($n>2$ should be $n\geq 2$ in the last part of proof of Proposition 8.1 of \cite{LV}). In particular, we have for $\tau_1$
\begin{equation}\label{iso}
\sharp S_{1}(E(\tau_1)/A)=\sharp S_{2}(E(\tau_1)/A)=\frac{1}{2}\sharp\Isom_{A/\mu}(E(\tau_1),E')=6
\end{equation}
and
$$\sharp\Isom_{A/\mu^2}(E(\tau_1),E')=\sharp S_{3}(E(\tau_1)/A).$$ 
The number $6$ is due to the Proof of Proposition 2.3 in 
\cite{GZ85} in the $l=3$ case.
In the remaining part of the proof we will show that $\sharp S_3(E(\tau_1)/A)=\sharp S_1(E(\tau_1)/A)=6$. From this and (\ref{valuation}), we can deduce that $j(27\tau_0)-j(\sqrt{-3})$ is divided by $\sqrt{-3}^9=3^4\sqrt{-3}$. But $j(27\tau_0)-j(\sqrt{-3})\in H_p$ is real and $\sqrt{-3}$ is the only purely imaginary generator in $H_p$ over $\BQ$, so $j(27\tau_0)-j(\sqrt{-3})$ should be divided by $3^5$. 

By \cite[Theorem 3.2]{LV}, if $(m,p)=1$, 
\begin{equation}\label{count}
\sum_{[\tau_1]}\sharp S_{n,m}(E(\tau_1)/A)=C\rho(m)\fU(3^{-n}m).
\end{equation}
Where $C=1$ if $4m=d_1d_2$ and $C=2$ otherwise. Here $\fU(\cdot)$ is the cardinality of a set of certain ideals of $\CO_{d_1}$ and $\rho(m)$ is a weight, for the definitions we refer to \cite{LV}. By \cite[Theorem 1.5 and Proposition 7.12]{LV},  if $(m,p)=1$, we have the following explicit formula,
$$\rho(m)\fU(3^{-n}m)=\begin{cases} 2\prod_{\ell|m, \ell\neq 3}\lrb{\sum_{i=0}^{\ord_{\ell}(m)}\lrb{\frac{-3}{\ell}}^i},\ \text{if}\  3^{-n}m\in \BZ_{>0}\ \text{and}\ 2\nmid m,\\ 0,\ \text{otherwise}, \end{cases}$$  
or equivalently, 
\begin{equation}\label{counting}
\rho(m)\fU(3^{-n}m)=\begin{cases} 2\sum_{0<r\mid m}\lrb{\frac{-3}{r}},\ \text{if}\  3^{-n}m\in \BZ_{>0}\ \text{and}\ 2\nmid m,\\ 0,\ \text{otherwise}. \end{cases}
\end{equation}
Here $\lrb{\frac{\cdot}{\cdot}}$ is the Kronecker symbol.

It is known that $S_3(E(\tau_1)/A)\subset S_{1}(E(\tau_1)/A)$ for any $\tau_1$. In order to prove $\sharp S_3(E(\tau_1)/A)=\sharp S_1(E(\tau_1)/A)$, we just need to prove that 
\begin{equation}\label{equ}
\sum_{[\tau_1]}\sharp S_3(E(\tau_1)/A)=\sum_{[\tau_1]}\sharp S_1(E(\tau_1)/A).
\end{equation}
By (\ref{iso}) and Proposition \ref{LCF}(3), we have that 
\[\sum_{[\tau_1]}\sharp S_1(E(\tau_1)/A)=6\sharp\Pic(\CO_{d_1})/2=p+1.\]
By (\ref{Sn}), (\ref{count}) and (\ref{counting}), we see that
\[\sum_{[\tau_1]}\sharp S_3(E(\tau_1)/A)=2\sum_{\substack{0<x<p\\ 2|x}}\sum_{0<r|p^2-x^2}\lrb{\frac{-3}{r}}+\sum_{[\tau_1]}\sharp S_{3,9p^2}((E(\tau_1)/A)).\]
Proposition \ref{sumKron} below shows that $S_{3,9p^2}((E(\tau_1)/A))=\emptyset$ and (\ref{equ}) is true. This finishes the proof of the theorem.
\end{proof}
\begin{remark}
In fact, one can show that $S_{1,m}((E(\tau_1)/A))=S_{3,m}((E(\tau_1)/A))$ if $m\neq 9p^2$ using the condition $m=\frac{36p^2-x^2}{4}$ and  (\ref{count})-(\ref{counting}) since $3|\frac{36p^2-x^2}{4}$ if and only if $3^3|\frac{36p^2-x^2}{4}$. So we also have $S_{1,9p^2}((E(\tau_1)/A))=\emptyset$.
\end{remark}

The proof of the following identities was communicated to the author by D. Zagier.   `GH from MO' on mathoverflow gave another proof of one special case using Siegel mass formula, see \cite{math}.  

\begin{proposition}\label{sumKron}
Let $p$ be a prime, then 
\[2\sum_{\substack{0<x<p\\ 2\nmid x}}\sum_{0<r\mid p^2-x^2}\lrb{\frac{-3}{r}}=\begin{cases}\frac{p-1}{3},\ \mathrm{if}\ p\equiv 1\mod 3,\\ \frac{p+1}{3},\ \mathrm{if}\ p\equiv 2\mod 3.\end{cases}\]
\[2\sum_{\substack{0<x<p\\ 2\mid x}}\sum_{0<r\mid p^2-x^2}\lrb{\frac{-3}{r}}=\begin{cases}p-3,\ \mathrm{if}\ p\equiv 1\mod 3,\\ p+1,\ \mathrm{if}\ p\equiv 2\mod 3.\end{cases}\]
Here $\lrb{\frac{\cdot}{\cdot}}$ is the Kronecker symbol.
\end{proposition}
\begin{proof}
Let $r_{-3}(n)$ be the representation number of $n$ by the quadratic form $a^2+ab+b^2$ and $R(n)=\sum_{0<d\mid n}\lrb{\frac{-3}{d}}$ if $n\neq 0$. We also let $R(0)=\frac{1}{6}$. Then it is well-known that $r_{-3}(n)=6R(n)$. The theta series
\[\theta(z)=\sum_{n\in\BZ}R(n)q^{n^2}=\frac{1}{6}\sum_{n\in \BZ}r_{-3}(n)q^{n^2}\]
is a modular form of $\Gamma_0(3)$ of weight $1$ with character $\lrb{\frac{-3}{\cdot}}$, see \cite[Page 30]{Zagier}. So $\theta^2(z)$ is a modular form of $\Gamma_0(3)$ of weight $2$. Let 
\[G_2(z)=-\frac{1}{24}+\sum_{n=1}^{\infty}\sigma(n)q^n\]
be the `Eisenstein series of weight $2$' (it is not a modular form actually), where $\sigma(n)=\sum_{0<d|n}d$ is the divisor function. However, 
\[G(z)=G_2(z)-3G_2(3z)\]
is really a modular form of $\Gamma_0(3)$ of weight $2$ \cite[Page 3]{Zagier}. By the dimension formula, the vector space of modular forms of $\Gamma_0(3)$ of weight $2$ has dimension $1$, so
\begin{equation}\label{theta1}
\theta(z)^2=\frac{1}{3}G(z)
\end{equation}
Now if $x$ is odd, $R(p^2-x^2)=R(\frac{p+x}{2})R(\frac{p-x}{2})$ since $R(n)$ is multiplicative and $(\frac{p+x}{2},\frac{p-x}{2})=1$. Then 
\begin{eqnarray*}
2\sum_{\substack{0<x<p\\ 2\nmid x}}R(p^2-x^2)+2R(0)R(p)=a_{p}
\end{eqnarray*}
where $a_{p}$ is the Fourier coefficient of $q^{p}$ in $\theta(z)^2$. But by (\ref{theta1}), $a_p=\frac{1}{3}\sigma(p)=\frac{p+1}{3}$. Then the first formula follows from the fact that $2R(0)R(p)=1/3$ or $0$ according to $p\equiv1$ or $2\mod 3$. 

If $2\mid x$, then $R(p^2-x^2)=R(p-x)R(p+x)$ since $(p-x,p+x)=1$. Then we have
\[2\sum_{\substack{0<x<p\\ 2\mid x}}R(p^2-x^2)+R(p)R(p)+2\sum_{\substack{0<x\leq p\\ 2\nmid x}}R(p-x)R(p+x)=a_{2p}=p+1,\]
where $a_{2p}$ is the Fourier coefficient of $q^{2p}$ in $\theta(z)^2$. If $2\nmid x$, then $\ord_2(p-x)=1$ or $\ord_2(p+x)=1$ which implies $R(p-x)R(p+x)=0$ by the definition of $R(n)$. So, we get
\[2\sum_{\substack{0<x<p\\ 2\mid x}}R(p^2-x^2)=p+1-R(p)^2=\begin{cases}p-3,\ \mathrm{if} \ p\equiv 1\mod 3,\\ p+1,\ \mathrm{if}\ p\equiv 2\mod 3.\end{cases}\]
\end{proof}

\section{Reduction of Heegner points}

Recall that $x(\tau)$ and $y(\tau)$ are the modular parametrization of $E^3$ in (\ref{modular}).
Using Sagemath we can compute that 
\[x\lrb{\frac{\omega}{9(2\omega+1)}}=4\sqrt{-3}\sqrt[3]{3}\omega^2,\ \ \  y\lrb{\frac{\omega}{9(2\omega+1)}}=24\omega-12.\]
Then by \cite[Proposition 3.1]{SSY}, we have
\[y\lrb{\frac{p\omega}{9(2\omega+1)}}\equiv y\lrb{\frac{\omega}{9(2\omega+1)}}^p\equiv 24\omega^2-12\mod p\]
and
\[x\lrb{\frac{p\omega}{9(2\omega+1)}}\equiv x\lrb{\frac{\omega}{9(2\omega+1)}}^p\equiv (4\sqrt{-3}\sqrt[3]{3}\omega^2)^p\mod p.\]
Since $\sqrt[3]{3}\in \BF_{p}$ by Lemma \ref{prime}, $(\sqrt[3]{3})^p\equiv \sqrt[3]{3}\mod p$. Also $(-3)^{\frac{p-1}{2}}=\left(\frac{-3}{p}\right)=-1\mod p$, then we get
\[x\lrb{\frac{p\omega}{9(2\omega+1)}}\equiv (-4\sqrt{-3}\sqrt[3]{3}\omega)\mod p.\]
That is 
\begin{equation}\label{congruence1}
P_{\tau_0}=\psi\left(\frac{p\omega}{9(2\omega+1)}\right)\equiv \left(-4\sqrt{-3}\sqrt[3]{3}\omega, 24\omega^2-12\right)\mod p.
\end{equation}

Write 
$$f(\tau)= \frac{\eta(\tau)^4}{\eta(3\tau)^4}=q^{-\frac{1}{3}}\prod_{n=0}^{\infty}(1-q^{3n+1})^4(1-q^{3n+2})^4,\ \ \ \tau\in \CH,$$
with $q=e^{2\pi i \tau}$.
Using \cite[Proposition 3.1]{SSY}, we also get
\begin{equation}\label{congruence2}
f(9\tau_0)\equiv \lrb{f\lrb{\frac{\omega}{(2\omega+1)}}}^p \equiv -3\omega^2 \mod p.
\end{equation}

\begin{lem}\label{keylem}
$f(9\tau_0)=3\omega^2 u$ for some real unit $u$.
\end{lem}
\begin{proof}
Firstly, we have $$f(9\tau_0)=f\lrb{\frac{p}{2}+\frac{p\sqrt{-3}}{6}}=\omega^2 r$$ with $r\in\BR$, since 
$$e^{2\pi i(\frac{p}{2}+\frac{p\sqrt{-3}}{6})}=-e^{\frac{-p\sqrt{3}}{3}2\pi}\ \text{and}\ e^{-\frac{2\pi i}{3}(\frac{p}{2}+\frac{p\sqrt{-3}}{6})}=\pm e^{-\frac{2\pi i}{3}}e^{\frac{p\sqrt{3}}{9}2\pi}.$$ 
By \cite[Theorem 12.2.4]{Lang-ef}, $f(9\tau_0)^6=\Delta(9\tau_0)/\Delta(27\tau_0)$ is an algebraic integer and  divides $3^{12}$. So $f(9\tau_0)$ is also algebraic integral and divides $9$. 

On the other hand, $f(\tau)^3$ is the Hauptmodul of $\Gamma_0(3)$ and it satisfies the modular equation \cite[Table 1]{Morain}:
\[(f(\tau)^3+27)(f(\tau)^3+3)^3=j(3\tau)f(\tau)^3.\]
So we get for $\tau=9\tau_0$,
\begin{equation}\label{modularequation}
\lrb{1+\frac{27}{f(9\tau_0)^3}}\lrb{\frac{f(9\tau_0)^3}{3}+1}^3=\frac{j(27\tau_0)}{27}.
\end{equation}
From Theorem \ref{singularmoduli} we know that $j(27\tau_0)/27$ is integral. So by (\ref{modularequation}) and the fact that $f(9\tau_0)\mid 9$, we have $3\mid f(9\tau_0)^3$  and $f(9\tau_0)^3\mid 27$. By Corollary \ref{deffield}, $f(9\tau_0)\in H_p$. Since $\sqrt{-3}$ is not ramified in $H_{p}$, $3\mid f(9\tau_0)^3$ implies $\sqrt{-3}\mid f(9\tau_0)$. But $f(9\tau_0)/\omega^2$ is a real number, so $3\mid f(9\tau_0)$. Together with the fact $f(9\tau_0)^3\mid 27$, we know that $f(9\tau_0)/(3\omega^2)$ is a real unit which finishes the proof.
\end{proof}

For a prime $\wp\mid\sqrt{-3}$ of $H_p$, $E^1\mod \wp$ is singular with the singular point $(-1,0)$. Let $E^1_0$ be the preimage of the nonsingular locus, then we have
\begin{proposition}\label{singular}
$$E^1(H_p)/E^1_0(H_p)=\{O, (-4,\pm 4\sqrt{-3})\}.$$ 
\end{proposition}
\begin{proof}
This can be deduced using the Tate algorithm, see \cite{Tate75} or \cite[IV.9]{Silvermanbook2}. To be more explicit, the order of $E^1(H_{p,\wp})/E^{1}_0(H_{p,\wp})$ is determined by the Kodaira type of $E/H_{p,\wp}$ where $H_{p,\wp}$ is the completion of $H_p$ at $\wp$. By Tate algorithm, we first change the equation of $E^1$ to the form $y^2=x^3-3x^2+3x+15$ such that $\pi\mid a_3,a_4,a_6$ where $\pi$ is a uniformizer of $H_{p,\wp}$ and $a_3,a_4,a_6$ are coefficients appearing in the general Weierstrass equation $y^2+a_1xy+a_3y=x^3+a_2x^2+a_4x+a_6$. Then we get $\pi\mid b_2=-6$, $\pi^2\mid a_6=15$, $\pi^3\mid b_8=-189$ and $\pi^3\nmid b_3=60$ where $b_2=a_1^2+4a_2$, $b_4=a_1a_3+2a_4$, $b_6=a_3^2+4a_6$ and $b_8=(b_2b_6-b_4^2)/4$.
Also the equation $T^2-\frac{15}{\pi^2}=0$ has solutions in $\BF_9$, the residue field of $\wp$, since $T^2-\frac{15}{(\sqrt{-3})^2}$ has solutions in $\BF_9$ and $\pi\mid \sqrt{-3}$. Then by step 5 of Tate algorithm, the order of $E^1(H_{p,\wp})/E^{1}_0(H_{p,\wp})$ is $3$. But $(-4, 4\sqrt{-3})$ and $(-4,-4\sqrt{-3})$ are all singular modulo $\wp$ and not in the same singular component, so  $E^1(H_{p,\wp})/E^{1}_0(H_{p,\wp})=\{O, (-4,\pm 4\sqrt{-3})\}$. Then 
$$E^1(H_p)/E^1_0(H_p)\cong E^1(H_{p,\wp})/E^{1}_0(H_{p,\wp})=\{O, (-4,\pm 4\sqrt{-3})\}.$$ 
\end{proof}

Denote $S=(-4,4\sqrt{-3})$, then $-S=(-4,-4\sqrt{-3})$. They are all primitive $3$-torsion points. In this section we prove that $\phi(P_{\tau_0})$ belongs to the same singular component as $S$ for any $\wp$ above $3$.

\begin{proposition}\label{iff1}
$3^5\mid(f(9\tau_0)^3-27)$ if and only if $3^5\mid (j(27\tau_0)-j(\sqrt{-3}))$.
\end{proposition}
\begin{proof}
Recall the equation in (\ref{modularequation}).
Since $\lrb{\frac{f(9\tau_0)^3}{3}+1}^3\equiv 1\mod 9$ by Lemma \ref{keylem}, $3^5\mid(f(9\tau_0)^3-27)$ if and only if $\frac{j(27\tau_0)}{27}-200\equiv 0\mod 9$, i.e., $j(27\tau_0)-5400\equiv 0\mod 3^5$. But $j(\sqrt{-3})=5400$, this finishes the proof.
\end{proof}

\begin{proposition}\label{iff2}
$9\mid\lrb{\frac{\eta(9\tau_0)^4}{\eta(27\tau_0)^4}-3\omega^2}$ if and only if $3^5\mid(f(9\tau_0)^3-27)$.
\end{proposition}
\begin{proof}
By Lemma \ref{keylem}, $f(9\tau_0)=3\omega^2 u$ for some real unit $u$. Then $3^5\mid (f(9\tau_0)^3-27)$ is equivalent to $9\mid (u^3-1)$ and $9\mid \lrb{\frac{\eta(9\tau_0)^4}{\eta(27\tau_0)^4}-3\omega^2}$ is equivalent to $3\mid (u-1)$.  Note that 
\begin{equation}\label{cube}
u^3-1=(u-1)((u-1)^2+3u).
\end{equation}
Hence $3\mid (u-1)$ obviously implies $9\mid (u^3-1)$. Conversely, if $9\mid (u^3-1)$ then for any $\wp\mid\sqrt{-3}$, we have $\wp\mid (u-1)$ by (\ref{cube}). So $\sqrt{-3}\mid (u-1)$. But $u-1\in H_p$ is a real number and $H_p=K(j(9\tau_0))$ with $j(9\tau_0)\in \BR$ by CM theory (see \cite[Page 133, Remark 1]{Lang-ef}). This forces $3\mid (u-1)$ and completes the proof.
\end{proof}

Now we can prove the following theorem about the reduction of our Heegner points.

\begin{thm}\label{singularcomponent}
$\phi(P_{\tau_0})-S\mod \wp$ is nonsingular for any prime $\wp$ above $\sqrt{-3}$ of $H_p$.
\end{thm}
\begin{proof}
Assume $\phi(P_{\tau_0})=(x_1,y_1)$. By Proposition \ref{iff1}, Proposition \ref{iff2} and Theorem \ref{singularmoduli}, we know that $9\mid\lrb{\frac{\eta(9\tau_0)^4}{\eta(27\tau_0)^4}-3\omega^2}$.
Then 
\begin{equation}\label{cong1}y_1-4\sqrt{-3}=-\frac{8}{3}\left(\frac{\eta(9\tau_0)^4}{\eta(27\tau_0)^4}-3\omega^2\right)\equiv 0\mod 3.
\end{equation}
By the equation of $E^1$, we have $y_1^2=x_1^3+16$, equivalently, 
\begin{equation}\label{eque1}
(x_1+4)((x_1+4)^2-12x_1)=(y_1+4\sqrt{-3})(y_1-4\sqrt{-3}).
\end{equation}
The right hand side of (\ref{eque1}) is divided by $3\sqrt{-3}$ by (\ref{cong1}), so $3\sqrt{-3}\mid (x_1+4)^3$ on the left hand side of (\ref{eque1}).
As a result, we have
\begin{equation}\label{cong2}x_1+4\equiv 0\mod \sqrt{-3}\end{equation}
and $\phi(P_{\tau_0}) \mod \wp$ is singular for any $\wp\mid\sqrt{-3}$.  We prove that $\phi(P_{\tau_0})+S$ is always singular modulo $\wp$. Let $\phi(P_{\tau_0})+S=(x_2,y_2)$, then by the addition formula,
\begin{equation}\label{cong3}
x_2=\lrb{\frac{y_1-4\sqrt{-3}}{x_1+4}}^2-(x_1-4).
\end{equation}
We have seen $x_1\neq -4$ by (\ref{congruence1}), so (\ref{eque1}) is equivalent to
\begin{equation*} \frac{y_1-4\sqrt{-3}}{x_1+4}=\frac{((x_1+4)^2-12x_1)}{y_1+4\sqrt{-3}}.\end{equation*}
Then we have $$\frac{y_1-4\sqrt{-3}}{x_1+4}\equiv 0\mod \sqrt{-3},$$
since by (\ref{cong1})(\ref{cong2}), $((x_1+4)^2-12x_1)$ is divided by $3$ but $y_1+4\sqrt{-3}$ can be divided just by $\sqrt{-3}$. As a result, $x_2\equiv -1\mod \sqrt{-3}$ by (\ref{cong2})(\ref{cong3}). Hence, $\phi(P_{\tau_0})+S\mod \wp$ is always singular and as a result, $\phi(P_{\tau_0})-S\mod\wp$ is always nonsingular by Proposition \ref{singular}. 
\end{proof}

\section{Nontriviality of Heegner points}\label{nontrivial}
Recall our Heegner points are
\[z=\tr_{H_p(\sqrt[3]{3})/L_{(3,p)}}P_{\tau_0}\in E^3(L_{(3,p)}),\]
with $\phi(z)\in E^1(L_{(p)})$.
Let \[R_{\tau_0}=\left(-4\sqrt{-3}\sqrt[3]{3}\omega, 24\omega^2-12\right)\in E^3(L_{(3)}),\]
then $R_{\tau_0}\equiv P_{\tau_0}\mod p$. Recall we have the isomorphism
\[\phi: E^3: y^2=x^3+16\cdot 9\ra E^1: y^2=x^3+16,\]
\[(x,y)\mapsto (x/\sqrt[3]{9}, y/3).\] 
Mapping $R_{\tau_0}+(0,-12)$ to $E^1$ through $\phi$,
we get $R_{\tau_0}'=(-\frac{4}{\sqrt[3]{3}},-\frac{4\sqrt{-3}}{3})$. Under the isomorphism
\[\phi_1: E^1:y^2=x^3+16\ra E_{min}^1:y^2+y=x^3,\]
\[(x,y)\mapsto \lrb{\frac{x}{4},\frac{y-4}{8}},\]
$R_{\tau_0}'$ is mapped to $a=(-\frac{1}{\sqrt[3]{3}}, \frac{-3-\sqrt{-3}}{6})$. Here we choose use the minimal model $E^1_{min}$ of $E^1$ in order to make the polynomial $D(x)$ below as simple as possible.
Since $3\mod p$ is always a cube by Lemma \ref{prime}, both $\phi$ and $\phi_1$ reduce to isomorphisms over $\BF_{p^2}$. Let $b=\phi_1(S)=(-1,\omega)$ be the primitive $3$-torsion point on $E^1_{min}$. Define the point $c=a-b=\left(-\frac{(\sqrt[3]{3})^2 + 3\sqrt[3]{3}+1}{4}, \frac{\sqrt{-3}(\sqrt[3]{3})^2+3\sqrt{-3}\sqrt[3]{3}+5\sqrt{-3}-4}{8}\right)$.

\begin{lem}\label{torsion}
We have $E^1(L_{(p)})_{\tor}=E^1(K)$.
\end{lem}
\begin{proof}
The proof is simpler than Lemma 3.6 of \cite{SSY}. Let $\ell$ be a prime. Suppose $P\in E^1[\ell]$ is a torsion point of order $\ell$. By \cite[Theorem 5.5]{Shimurabook}, the field $K(P)$ contains a subfield over $K$ with Galois group $(\CO_K/\fl)^\times$, where $\fl$ is a prime ideal of $K$ above $\ell$. If $\ell\neq 2,3$,  then $[K(P):K]>3$; 

For $\ell=2$, we see that $K(E^1[2])/K$ has degree $3$,  and is unramified at $p$. Thus $$K(E^1[2])\cap L_{(p)}\subset K.$$ 

Suppose $P\in E^1(L_{(p)})[3^\infty]$. Since $E^1$ has good reductions outside $3$, $K(P)$ is unramified outside $3$, which forces that $K(P)\subset K$. We see that 
$$E^1(L_{(3,p)})[3^\infty]=E^1[3]=E^1(K).$$
\end{proof}

\begin{proposition}\label{mainthm}
Let $c$ be the point on $E^1_{min}$ defined above. If $\left[\frac{p+1}{9}\right]c\neq O, \pm (0,-1) \mod p$,
then $z$ is nontorsion.
\end{proposition}
\begin{proof}
By Theorem \ref{singularcomponent},
we can write $\phi(P_{\tau_0})=P'+S$ where $P' \mod \wp$ is nonsingular for any prime $\wp$ above $3$. Now
\begin{equation}\label{P+S}
\phi(z)=\tr_{H_p/L_{(p)}}\phi(P')+\left[\frac{p+1}{9}\right]S. 
\end{equation}
Assume $\left[\frac{p+1}{9}\right]c\neq O,\pm (0,-1) \mod p$, then $\left[\frac{p+1}{9}\right](R_{\tau_0}-S)\neq O,\pm (0,-12)\mod p$. So 
\begin{equation}\label{onethree}
\phi(z)\equiv \left[\frac{p+1}{9}\right]R_{\tau_0}\neq \left[\frac{p+1}{9}\right]S, \left[\frac{p+1}{9}\right]S\pm (0,-12)\mod p. 
\end{equation}
If $z$ is torsion, then $\phi(z)$ is torsion. By Lemma \ref{torsion}, $\phi(z)$ is $3$-torsion. By (\ref{P+S}), $\phi(z)$ is on the same singular component with $\left[\frac{p+1}{9}\right]S$ for any $\wp$. Then $\phi(z)$ should be one of $\left[\frac{p+1}{9}\right]S$ and $\left[\frac{p+1}{9}\right]S\pm (0,-12)$ which contradicts (\ref{onethree}). So $z$ should be nontorsion.
\end{proof}
We want to translate the conditions in Proposition \ref{mainthm} into conditions about polynomials. For this, we need some preparation.
For a general point $(x,y)\in E^1_{min}(\bar K)$, by the addition formula on \cite[Page 54]{silvermanbook1}(using $\sqrt{-3}=\omega-\omega^2$) and \cite[Exercise 3.7]{silvermanbook1}, we have
\begin{equation}\label{add1}
[\sqrt{-3}](x,y)=\lrb{\frac{x^3+1}{-3x^2},\ -\frac{(2y+1)(x^3+1)}{3\sqrt{-3}x^3}-\frac{-y+\omega}{\sqrt{-3}}-1}
\end{equation}
and
\begin{equation}\label{add2}
[3](x,y)=\lrb{\frac{x^9-24x^6+ 3x^3+1}{9x^2(x^3+1)^2},\ \frac{y\omega_3(x)}{\psi_3(x)}},
\end{equation}
where $\omega_3(x)$ and $\psi_3(x)$ are polynomials in $x$. Assume $P_1=(x_1,y_1)\in E^1_{min}(\bar K)$ satisfies $[\sqrt{-3}]P_1=c$. By (\ref{add1}), we know that $x_1$ is the root of the polynomial
\[x^3-3\frac{(\sqrt[3]{3})^2 + 3\sqrt[3]{3}+1}{4}x^2+1=0.\]
Using Sagemath we can compute that 
$P_1=(\sqrt[3]{3}-1, -(\sqrt[3]{3})^2+1)+[i](0,0)$ for $i=0,1,2$ where $(0,0)$ is a $\sqrt{-3}$-torsion point. 

\begin{proposition}\label{more}
For $n\in \BN$, $c$ is divisible by $3^n\sqrt{-3}$ in $E^1_{min}(\BF_{p^2})$ if and only if $c$ is divisible by $3^{n+1}$ in $E^1_{min}(\BF_{p^2})$.
\end{proposition}
\begin{proof}
The if part is obvious. We prove the only if part. Let $v$ be the $3$ order of $p+1$. It is well-known that in our case 
\begin{equation}\label{finite1}
E^1_{min}(\BF_{p})\simeq\BZ/(p+1)\BZ\simeq(\BZ/((p+1)/3^v)\BZ)\times (\BZ/3^v\BZ)
\end{equation}
and
\begin{equation}\label{finite}
E^1_{min}(\BF_{p^2})\simeq\BZ/(p+1)\BZ\times\BZ/(p+1)\BZ\simeq(\BZ/((p+1)/3^v)\BZ)^2\times (\BZ/3^v\BZ)^2,
\end{equation} 
see for example \cite{Wittmann} and the reference therein. As a result, we have
\begin{equation}\label{finite2}
E^1_{min}(\BF_{p^2})/E^1_{min}(\BF_{p})\simeq\BZ/(p+1)\BZ.
\end{equation}

We also have 
\begin{equation}\label{torsion1}
E^1_{min}(\BF_{p^2})[3]=[\BZ/3\BZ](0, -1)\oplus [\BZ/3\BZ](-1,\omega)
\end{equation}
where $(0,-1)$ is $\sqrt{-3}$-torsion and $(-1,\omega)$ is primitive $3$-torsion.
Let $c=[3^n\sqrt{-3}]d$ in $E^1_{min}(\BF_{p^2})$. If $c$ is not divisible by $3^{n+1}$, then $d$ is not divisible by $\sqrt{-3}$. Then by (\ref{finite}) and (\ref{torsion1}), $[\frac{p+1}{3}]d=\pm(-1,\omega)$. As a result, 
$[\frac{p+1}{3^{n+1}}]c\neq 0$ while $[\sqrt{-3}][\frac{p+1}{3^{n+1}}]c=0$. But this is impossible. The reason is as follows: Since $P_1\mod p\in E^1_{min}(\BF_p)$, $[\frac{p+1}{3^{n+1}}]P_1\in E^1_{min}(\BF_p)$. By (\ref{add1}), 
$[\frac{p+1}{3^{n+1}}]c=[\sqrt{-3}][\frac{p+1}{3^{n+1}}]P_1$ belongs to $E^1_{min}(\BF_{p^2})\backslash E^1_{min}(\BF_p)$.
\end{proof}


Now, we investigate the $9$-divisibility of $c$ in $E^1_{min}(\BF_{p^2})$. Let $$D(x)=x^9-24x^6+ 3x^3+1-9(\sqrt[3]{3}-1)x^2(x^3+1)^2,$$
which is irreducible over $L_{(3)}$.


\begin{proposition}\label{iff}
$\left[\frac{p+1}{9}\right]c\neq O,\pm(0,-1) \mod p$ if and only if $c$ cannot be divided by $9$ in $E^1_{min}(\BF_{p^2})$.
\end{proposition}
\begin{proof}Recall that we list the structure of $E^1_{min}(\BF_{p^2})$ in (\ref{finite1})-(\ref{torsion1}).
Let $e\in E^1_{min}(\BF_{p^2})\backslash E^1_{min}(\BF_{p})$ be the generator of $E^1_{min}(\BF_{p^2})/ E^1_{min}(\BF_{p})$. Then $e$ is of order $p+1$ and $c\mod p=ne$ for some integer $n$. Then $[\frac{p+1}{3}]e$ also belongs to $E^1_{min}(\BF_{p^2})\backslash E^1_{min}(\BF_{p})$ and should be of primitive order $3$. By Proposition \ref{more} and the discussion before Proposition \ref{aD}, $c\mod p$ can always be divided by $3$. If $c\mod p$ cannot be divided by $9$, we can assume $n=3m$ with $(3, m)=1$. Then $[\frac{p+1}{9}]c= [m][\frac{p+1}{3}]e$ is primitive of order $3$ which cannot be $O$ or $\pm(0,-1)$. Conversely, if $c\mod p$ is divisible by $9$, then $9\mid n$ and $[\frac{p+1}{9}]c\equiv [p+1]\left[\frac{n}{9}\right]e=O\mod p$. This finishes the proof of the proposition.
\end{proof}

\begin{proposition}\label{aD}
$c\mod p$ is divisible by $9$ if and only if $D(x) \mod p$ has solutions in $\BF_p$.
\end{proposition}
\begin{proof}
By Proposition \ref{more}, we know that $c\mod p$ is divisible by $9$ if and only if $c$ is divisible by $3\sqrt{-3}$ in $E^1_{min}(\BF_{p^2})$, i.e., if and only if $P_1$ is divisible by $3$ in $E^1_{min}(\BF_{p^2})$. Since $(0,0)$ is $\sqrt{-3}$-torsion and $9\mid p+1$, $(0,0)$ can always be divided by $3$ in $E^1_{min}(\BF_{p^2})$ by the structure in (\ref{finite}). Then $P_1$ is divisible by $3$ in $E^1_{min}(\BF_{p^2})$ if and only if $(\sqrt[3]{3}-1, -(\sqrt[3]{3})^2+1)$ is divisible by $3$ in $E^1_{min}(\BF_{p^2})$, i.e., if and only if $D(x)$ has solutions in $\BF_{p^2}$ by (\ref{add2}). We will prove $D(x)$ has solutions in $\BF_{p^2}$ if and only if $D(x)$ has solutions in $\BF_{p}$.

The direction `$D(x)$ has solutions in $\BF_{p}$ $\Rightarrow$ $D(x)$ has solutions in $\BF_{p^2}$' is clear. Now we prove the converse direction. So, we assume $D(x)$ has a solution in $\BF_{p^2}$. Since $E^1_{min}[3]\hookrightarrow E^1_{min}(\BF_{p^2})$, if $D(x)$ has one solution in $\BF_{p^2}$, then $D(x)$ splits totally in $\BF_{p^2}$, since different 3-division point of $P_1$ only differ by an element in $E^1_{min}[3]$. Because $D(x)$ always has a solution $x_0$ in $\BR$ which is an algebraic integer, we see $x_0 \mod p$ will give a solution of  $D(x) \mod p$ in $\BF_p$ since $\BF_{p^2}=\BF_p(\sqrt{-3})$.

\end{proof}

\begin{thm}\label{mc}
If $D(x)$ does not have solutions in $\BF_p$, then $z$ is nontorsion. 
\end{thm}
\begin{proof}
The theorem follows from Proposition \ref{mainthm}, Proposition \ref{iff}  and Proposition \ref{aD}.
\end{proof}

Let $\sigma_p\in\Gal(K(\sqrt[3]{p})/K)$ such that $\sigma_p(\sqrt[3]{p})=\omega\sqrt[3]{p}$ and let
\begin{equation}\label{R1}
z_1=z+[\omega^2]z^{\sigma_{p}}+[\omega]z^{\sigma_{p}^2},\ \ \
z_2=z+[\omega]z^{\sigma_{p}}+[\omega^2]z^{\sigma_{p}^2}.
\end{equation}
Then 
\begin{equation}\label{z1z2}
z_1^{\sigma_{p}}=[\omega]z_1,\ \ \ z_2^{\sigma_{p}}=[\omega^2]z_2.
\end{equation}
Let $z_t=z+z^{\sigma_{p}}+z^{\sigma_{p}^2}$. Since $(z_t)^{\sigma_{p}}=z_t$, $z_t$ corresponds to a point in $E^1(K)$ which should be torsion since $E^1(K)$ is of rank $0$. As a result, $z_t$ is torsion and in fact $0$ since $z_t=\left[\frac{p+1}{3}\right]\phi_1^{-1}(c)\mod p$ is $0$ (recall that $c$ can always be divided by $3$ in $E_{min}^1(\BF_{p^2}$)). So we have
\begin{equation}\label{sumz}
z_1+z_2=[3]z.
\end{equation}
\begin{thm}
Let $p\equiv 8\mod 9$ be a prime.  If $D(x)$ does not have solutions in $\BF_p$, then at least one of $p$ and $p^2$ can be written as the sum of two rational cubes.
\end{thm}
\begin{proof}
By (\ref{z1z2}), $z_1$ and $z_2$ can be twisted to $E^p(K)$ and $E^{p^2}(K)$ respectively. 
But at least one of $z_1$ and $z_2$ is nontorsion by Theorem \ref{mc} and (\ref{sumz}). Then the theorem follows from the well-known fact that ${\rm rank}_{\BZ}E^p(\BQ)={\rm rank}_{\CO_K}E^p(K)$ and ${\rm rank}_{\BZ}E^{p^2}(\BQ)={\rm rank}_{\CO_K}E^{p^2}(K)$, for a proof of this fact see \cite[Lemma 3.10]{SSY}.
\end{proof}

\begin{remark}
Since if $D(x)$ has solutions in $\BF_p$, then $D(x)$ will split in $\BF_{p^2}$, the condition $D(x)$ has solutions in $\BF_p$ is equivalent to that $p$ splits totally in the splitting field of $D(x)$. By the Chebotarev Density Theorem, the primes such that $D(x)$ has solutions in $\BF_p$ have positive density in all the primes congruent to $8$ modulo $9$. Numerical computation shows that the density is nearly $1/3$. So our main theorem covers nearly $2/3$ primes which are congruent to $8$ modulo $9$.
\end{remark}

\section{The explicit Gross-Zagier formulae}
Let $\pi$ be the automorphic representation of $\GL_2(\BA_{\BQ})$ corresponding to ${E^3}{/\BQ}$. Then $\pi$ is only ramified at $3$ with conductor $243$. For $n\in \BQ^\times$, let $\chi_n: \Gal(K^{\ab}/K)\rightarrow\BC^\times$
be the cubic character given by $\chi_n(\sigma)=(\sqrt[3]{n})^{\sigma-1}$. Define
\[L(s,E^3,\chi_n):=L(s-1/2,\pi_K\otimes \chi_n)\]
where $\pi_K$ is the base change of $\pi$ to $\GL_2(\BA_K)$.
We put $\chi=\chi_{3p^2}$ and $\chi'=\chi_{3p}$. By the Artin formalism, we have
$$L(s,E^3,\chi)=L(s,E^p)L(s,E^{9p^2}),\ \ \ L(s,E^3,\chi')=L(s,E^{p^2})L(s,E^{9p}) .$$
As in \cite[Proposition 4.1]{HSY19}, we can prove that the incoherent quaternion algebras over $\BA_{\BQ}$ which satisfies the Tunnell-Saito conditions for $(E^3,\chi)$ and $(E^3,\chi')$ are only ramified at infinity. So our basic settings are the same as in \cite{HSY19}, and the proofs follow the same line which we only sketch here. 

Let $f=\psi:X_0(243)\ra E^3$ be the modular parametrization in (\ref{modular})(we change our notation in order to keep in line with previous works). Then $f$ can be viewed as the newform in $\pi$ (for details see \cite[Page 6918]{HSY19}) and we denote $f_3$ its $3$-part. Define the Heegner cycle
$$P^0_{\chi}(f)=\frac{\sharp \Pic(\CO_p)}{\Vol(K^\times\widehat{\BQ}^\times\backslash \widehat{K}^\times)}
\int_{K^\times\widehat{\BQ}^\times\backslash \widehat{K}^\times}f(P_
{\tau_0})^{\sigma_t}\chi(t)dt,$$
and define $P^0_{\chi^{-1}}(f)$ similarly as in \cite[Theorem 1.6]{CST14}.
 
Let $f'\neq 0$ be a test vector in $V(\pi,\chi)$ which is defined in \cite[Definition 1.4]{CST14} (roughly speaking, it is some $\chi^{-1}$-eigenvector such that the period integral below is nonzero). In our case, the newform $f$ only differs from $f'$ at the local place $3$. For $f_3$, we define
\begin{equation}\label{Eq:Sec4WaldsIntNormalised}
 \beta_3^0(f_3,f_3)=\int_{\BQ_3^\times\backslash K_3^\times}\frac{(\pi(t)f_3,f_3)_3\chi_3(t)}{(f_3,f_3)_3}dt,
\end{equation}
here $(\cdot,\cdot)_3$ is a $\GL_2(\BQ_3)$-invariant pairing on $\pi_3\times \pi_3$ and $K^\times$ is embedded into $\GL_2(\BQ)$ as in (\ref{emb2}). Similarly, one can define $\beta_3^0(f'_3,f'_3)$.

\subsection{The explicit formulae} For integer  $n$,  let $\Omega^{n}$ be the minimal positive real period of $E^n$ defined by
\[\Omega^{n}=\int_{E_{min}^n(\BR)}|\omega^{n}|\]
where $\omega^{n}$ is the invariant differential on the minimal model $E_{min}^n$ of $E^n$. This is the period appearing in the BSD conjecture. However, Stephens studied the period $\Omega_{E^n}$ of $E^n$ defined using the invariant differential $\omega_{E^n}=dx/y$ on $E^n$, see \cite[Lemma 5 in section 5]{Stephens}. To relate $\Omega^n$ and $\Omega_{E^n}$, we need the following Lemma. 
\begin{lem}\label{min}
Let $n$ be a cube-free integer, the minimal Weierstrass equation of $E^n: y^2=x^3+16n^2$ is given as follows:
\begin{enumerate}
\item $Y^2=X^3+\frac{n^2}{4}$ if $2\mid n$;
\item $Y^2+Y=X^3+\frac{n^2-1}{4}$ if $2\nmid n$.
\end{enumerate}
and the transformations are given by $x=2^2X, y=2^3Y$ and $x=2^2X, y=2^3Y+4$ respectively.
\end{lem}
\begin{proof}
If $2\mid n$, the discriminant of the Weierstrass equation in $(1)$ is $-3^3\frac{n^4}{4}$ which is 12th power-free. If $2\nmid n$, the discriminant of the Weierstrass equation in $(2)$ is $-3^3n^4$ which is also 12th power-free. Then by \cite[Page 186, Remark 1.1]{silvermanbook1}, both of the equations in $(1)$ and $(2)$ are minimal.
\end{proof}
\begin{remark}
The proof of Lemma \ref{min} is inspired by \cite[Lemma 1]{Jed} which describes the minimal Weierstrass equation of $E_n: y^2=x^3-432n^2$. 
\end{remark}

\begin{proposition}\label{period}
We have 
\begin{equation*}\label{equation4}
\Omega^{p}\Omega^{9p^2}=\Omega^{p^2}\Omega^{9p}=(p)^{-1}(\Omega^3)^2.
\end{equation*}
\end{proposition}
\begin{proof}
By Lemma \ref{min},  $\omega^n=2\omega_{E^n}$. Since $E^n_{min}(\BR)=E^n(\BR)$, we have $\Omega^n=2\Omega_{E^n}$. Then the relations in the proposition come from the formula for $\Omega_{E^n}$ in
\cite[Lemma 5 in section 5]{Stephens}.
\end{proof}
Now we give the main theorem in this section.
\begin{thm} \label{thm:GZ}
We  have the following explicit formulae of Heegner points:
\[\frac{L'(1,E^{p})L(1,E^{9p^2})}{\Omega^{p}\Omega^{9p^2}}=3 \cdot \wh{h}_\BQ(z_1),\]
\[\frac{L'(1,E^{p^2})L(1,E^{9p})}{\Omega^{p^2}\Omega^{9p}}=3 \cdot \wh{h}_\BQ(z_2).\]
\end{thm}
\begin{proof}We only sketch the proof for the first formula, the proof of the second one is the same.
By \cite[Theorem 1.6]{CST14}, we have
\begin{equation}\label{derivative}
L'(1,E^3,\chi)=\frac{8\pi^2(\varphi,\varphi)_{\Gamma_0(243)}}{2\sqrt{3}p}\cdot\frac{\left\langle P^0_{\chi}(f), P^0_{\chi^{-1}}(f)\right\rangle_{K,K}}{(f,f)_{\CR'}}\cdot \frac{\beta_3^0(f'_3,f'_3)}{\beta_3^0(f_3,f_3)},
\end{equation}
where $\CR'$ is the admissible order of $\M_2(\wh{\BZ})$ for the pair $(\pi,\chi)$ (\cite[Definition 1.3]{CST14}) and $(\cdot,\cdot)_{\CR'}$ is the pairing on $\pi\times \pi^\vee$ defined as in \cite[page 789]{CST14}, and $\langle\cdot,\cdot\rangle_{K,K}$ is a pairing from $E^3(\ov{K})_\BQ\times_K E^3(\ov{K})_\BQ$ to $\BC$  such that $\langle\cdot,\cdot\rangle_{K}=\rm{Tr}_{\BC/\BR}\langle\cdot,\cdot\rangle_{K,K}$ is the Neron-Tate height over the base field $K$, see \cite[page 790]{CST14}. Finally, $\varphi$ is the cusp form of level $243$ and weight $2$ associated to $E^3$, and  $(\varphi,\varphi)_{\Gamma_0(243)}$ is the Petersson norm of $\varphi$. Let $\omega_{\varphi}=2\pi i \varphi(\tau)d\tau$ be the invariant differential on $X_0(243)$, then $f^*\omega^3=\omega_{\varphi}$, see \cite[Page 310]{GZ1986}.

\begin{enumerate}
\item[(1)] 
Using Sagemath we compute that $\{\Omega^{3},\Omega^{3}\cdot(\frac{1}{2}+\frac{\sqrt{-3}}{2})\}$ is a $\BZ$-basis of the period lattice $L$ of $E^3_{min}$. Under the isomorphism    $\BC/L\cong E^3_{min}(\BC)$, the pull back of $\omega^3$ is $dz$. So
\begin{eqnarray*}
\sqrt{3}(\Omega^3)^2=2\int_{\BC/L}dxdy&=&\int_{E^3_{min}(\BC)}|\omega^3\wedge\overline{\omega^3}|\\
 &=&\frac{1}{\deg f}\int_{X_0(243)(\BC)}|f^*\omega^3\wedge \ov{f^*\omega^3}|\\
 &=&\frac{1}{9}\cdot 8\pi^2(\varphi,\varphi)_{\Gamma_0(243)}.
\end{eqnarray*}
Here $\deg f$ is the degree of $f$ as the modular parametrization. By Proposition \ref{period}, we have
\begin{equation*}\label{equation4}
\Omega^{p}\Omega^{9p^2}=\Omega^{p^2}\Omega^{9p}=(p)^{-1}(\Omega^3)^2=\frac{8\pi^2(\varphi,\varphi)_{\Gamma_0(243)}}{9\sqrt{3}p}.
\end{equation*}

\item[(2)] In our case, by \cite[Lemma 2.2 and Lemma 3.5]{CST14},
$$(f,f)_{\CR'}=\deg f\cdot \frac{\Vol(\CR^\times)}{\Vol(\CR^{'\times})}=6.$$
Here $\CR$ is the standard Eichler order of $\M_2(\wh{\BZ})$ of discriminant $243$.
\item[(3)] Finally, we deal with the Heegner cycle.
$$P^0_{\chi}(f)=\frac{\# \Pic(\CO_p)}{\# \Pic(\CO_{9p})}\sum_{t\in\Pic(\CO_{9p})}f(P_{\tau_0})^{\sigma_t}\chi(t)=\frac{1}{9}\sum_{t\in\Pic(\CO_{9p})}f(P_{\tau_0})^{\sigma_t}\chi(t).$$
Let
\[z'_1=\sum_{\sigma\in \Gal(H_{9p}/L(9p))} f(P_{\tau_0})^\sigma\chi(\sigma)=3z_1\in E^3(L_{(9p)}),\]
where $L_{(9p)=K(\sqrt[3]{9p})}$. Then using the fact that $\langle\cdot ,\cdot\rangle_{K,K}$ is bilinear, symmetric and Galois invariant, we can show that
\begin{align*}\langle P^0_{\chi}(f),P^0_{\chi^{-1}}(f)\rangle_{K,K}&=\frac{1}{9^2}
\left\langle\sum_{\sigma\in\Gal(L_{(9p)}/K)}(z'_1)^{\sigma}\chi(\sigma),\sum_{\sigma\in\Gal(L_{(9p)}/K)}
(z'_1)^{\sigma}\chi^{-1}(\sigma)\right\rangle_{K,K}\\
&=\frac{1}{27}\left(\langle z'_1,z'_1\rangle_{K,K}-\left\langle z'_1,(z'_1)^{\sigma'}\right\rangle_{K,K}\right),\notag
\end{align*}
where $\sigma'$ is a generator of $\Gal(L_{(9p)}/K)$, for more explicit details see \cite[(4.6)]{HSY19}. 
Without loss of generality, we can assume $(z'_1)^{\sigma'}=[\omega]z'_1$, then
$$\left\langle z'_1,(z'_1)^{\sigma'}\right\rangle_{K,K}=\frac{1}{2}\left(\widehat{h}_K([1+\omega]z'_1)
-\widehat{h}_K([\omega]z'_1)-\widehat{h}_K(z'_1)\right),$$
where $\widehat{h}_K$(resp. $\widehat{h}_\BQ$ below) means the Neron-Tate height over $K$(resp. $\BQ$).
Since $|1+\omega|=|\omega|=1$, by definition, $\widehat{h}_K([1+\omega]z'_1)=\widehat{h}_K([\omega]z'_1)=\widehat{h}_K(z'_1)$. Then
$$
\left\langle z'_1,(z'_1)^{\sigma'}\right\rangle_{K,K}=-\frac{1}{2}\widehat{h}_K(z'_1),$$
and hence
\begin{equation*}\label{equation2}
\left\langle P^0_{\chi}(f),P^0_{\chi^{-1}}(f)\right\rangle_{K,K}=\frac{1}{18}\widehat{h}_K(z'_1)=\frac{1}{9}\widehat{h}_\BQ(z'_1)=\wh{h}_\BQ(z_1).
\end{equation*}
\end{enumerate}

Combining (\ref{derivative}) and (1)-(3) and Corollary \ref{ration} below, we get
\[\frac{L'(1,E^{p})L(1,E^{9p^2})}{\Omega^{p}\Omega^{9p^2}}=3 \cdot \wh{h}_\BQ(z_1).\]
\end{proof}

\subsection{Local Waldspurger periods} Again, the treatment of our local periods is the same with \cite[section 4]{HSY1}. So we only list the local arithmetic information about $E^3$, $\chi$ and $\chi'$,  but refer the proofs to the corresponding results in \cite[section 4]{HSY1}. Let $\Theta$ be the unitary Hecke character corresponding to $E^3$ through the CM theory.
Since $p$ is a cube in $K_3$, the $3$-parts of $\chi$ and $\chi'$ are the same and we denote them uniformly by $\chi_3$ in this subsection and its complex conjugate by $\ov\chi_3$.  We will use $c(\cdot)$ to denote the order of the conductor of a local character.
\begin{lem}\label{thetavalue}
We have $c(\Theta_3)=c(\chi_3)=4$. Their values are
given explicitly by
\begin{center}
\begin{tabular}{|c|c|c|c|c|c|}
\hline
&$-1$&$1+\sqrt{-3}$&$1-\sqrt{-3}$&$1+3\sqrt{-3}$&$\sqrt{-3}$\\
\hline
$\Theta_3$&$-1$&$\omega$&$\omega^2$&$\omega^2$&$i$\\
\hline
$\chi_3$ &$1$&$1$&$1$&$\omega^2$&$1$ \\ 
\hline
$\Theta_3\ov\chi_3$&$-1$&$\omega$&$\omega^2$&$1$&i\\
\hline
\end{tabular}
\end{center}
\end{lem}
\begin{proof}
The proof is exactly the same as \cite[Lemma 4.1 and Lemma 4.2]{HSY1}.


\end{proof}

Let $\theta_3$ be the $3$-adic character which parametrizes the supercuspidal representation $\pi_3$ via compact-induction construction as in \cite[Section 2.2]{HSY1}. The test vector issue for Waldspurger's period integral is closely related to $c(\theta_3\ov\chi_3)$ or $c(\theta_3\chi_3)$. We can work out these using Lemma \ref{thetavalue} and the relation between $\theta_3$ and $\Theta_3$ in \cite[Theorem 2.8]{HSY1}. Let $\psi_3$ be the additive character such that $\psi_3(x)=e^{2\pi i \iota(x)}$ where $\iota:\BQ_3\rightarrow \BQ_3/\BZ_3 \subset \BQ/\BZ$ is the map given by $x\mapsto -x\mod \BZ_3$. Let $\psi_{K_3}(x)=\psi_3\circ \Tr_{K_3/\BQ_3}(x)$ be the additive character of $K_3$. Denote $\alpha_{\theta_3\ov\chi_3}$  the number such that $\theta_3\ov\chi_3(1+u)=\psi_{K_3}(\alpha_{\theta_3\ov\chi_3}u)$ for any $u\in \sqrt{-3}\CO_{K,3}$ whose existence follows from \cite[Lemma 2.1]{HSY1}.

\begin{lem}
We have $c(\theta_3\ov\chi_3)=2$ and $\alpha_{\theta_3\ov\chi_3}=\frac{1}{3\sqrt{-3}}$. Moreover,  $c(\theta_3\ov{\chi}_3)< c(\theta_3\chi_3)$. 
\end{lem}

\begin{proof}
The proof is the same as \cite[Lemma 4.5]{HSY1}
\end{proof}

\begin{prop}\label{Prop:TestingForNew}
Suppose $\Vol(\BZ_3^\times\backslash\CO_{K,3}^\times)=1$ so that $\Vol(\BQ_3^\times\backslash K_3^\times)=2$.
For $f_3$ being the newform, $K$ being embedded in $\M_2(\BQ)$ as in $(\ref{emb2})$, we have
\begin{equation*}
 \beta^0_3(f_3,f_3)=1/2.
\end{equation*}
\end{prop}
\begin{proof}
The proof is the same as \cite[Proposition 4.1]{HSY1}.
\end{proof}
\begin{coro}\label{ration}
For the admissible test vector $f'_3$ and the newform $f_3$ we have 
\[\frac{\beta_3^0(f'_3,f'_3)}{\beta_3^0(f_3,f_3)}=4.\]
\end{coro}
\begin{proof}
Keep the normalization of the volumes as in Proposition \ref{Prop:TestingForNew}. From the definition of $f'$, we have $\beta_3^0(f'_3,f'_3)=\Vol(\BQ_3^\times\backslash K_3^\times)=2$. Then the corollary follows from Proposition \ref{Prop:TestingForNew}.
\end{proof}

\subsection{Applications}\label{app}
The explicit Gross-Zagier formulae can help us determine our Heegner points give solutions to which equations. Assuming the BSD conjecture, both $L'(1,E^{p})$ and $L'(1,E^{p^2})$ should be nonzero. By Theorem \ref{thm:GZ}, if $L(1,E^{9p^2})\neq 0$ and $L(1,E^{9p})=0$, then $z_2$ is torsion while $z_1$ is nontorsion and gives a rational point on $E^p$; if $L(1,E^{9p})\neq 0$ and $L(1,E^{9p^2})=0$, then $z_1$ is torsion while $z_2$ is nontorsion and gives a rational point on $E^{p^2}$; if both $L(1,E^{9p^2})$ and $L(1,E^{9p})$ are nonzero, then both $z_1$ and $z_2$ are nontorsion and give rational points on $E^{p}$ and $E^{p^2}$ respectively; if both $L(1,E^{9p^2})$ and $L(1,E^{9p})$ are zero, then both $z_1$ and $z_2$ are torsion points and cannot give rational points on either $E^p$ or $E^{p^2}$. The numerical computation shows all the cases for the central $L$-values can happen.

\section{An example}
Let $E^3_{min}: y^2+y=x^3+2$ be the minimal model of $E^3$. The map $\phi_3: E^3_{min}\ra E^3$ is given by $(x,y)\mapsto (4x,8y+4)$. Since the modular parametrization in Sagemath is designed for minimal models, we first compute the point using $E^3_{min}$ and then map it to $E^3$.

We take $p=17$ and $\tau_0=17\omega/(9(2\omega+1))$. Let $P_{\tau_0}=\psi_{min}(\tau_0)\in E^3_{min}(H_{9p})$. Then

{ \begin{eqnarray*}
x(P_{\tau_0})&=&\left\{\left(\left(\frac{209}{204}\sqrt{-3} + \frac{145}{68}\right)\sqrt{17} + \frac{17}{4}\sqrt{-3} + \frac{35}{4}\right)\sqrt[3]{17^2}\right.\\ 
&&\ \ \ \ \ + \left(\left(\frac{8}{3}\sqrt{-3} + \frac{11}{2}\right)\sqrt{17}+ 11\sqrt{-3} + \frac{45}{2}\right)\sqrt[3]{17} \\ 
&&\ \ \ \ \ \ \ \ \ +\left. \left(\frac{41}{6}\sqrt{-3} + 14\right)\sqrt{17} + 28\sqrt{-3} + \frac{117}{2}\right\}\cdot\sqrt[3]{3^2},
\end{eqnarray*}}
\begin{eqnarray*}
y(P_{\tau_0})&=&\left(\frac{6045}{34}\omega^2\sqrt{17}+\frac{1467}{2}\omega^2\right)\sqrt[3]{17^2} + \left(\frac{915}{2}\omega^2\sqrt{17} + \frac{3771}{2}\omega^2\right)\sqrt[3]{17}\\   &&\ \ \ \ \ \ \ \ \ \ \ +1176\omega^2\sqrt{17}+4848\omega^2-2.
\end{eqnarray*}

We see that $P_{\tau_0}$ is defined over $H_{17}(\sqrt[3]{3})$ as predicted by our Corollary \ref{deffield} and $H_{17}=K(\sqrt[3]{17},\sqrt{17})$. Take the trace of $P_{\tau_0}$ from $H_{17}$ to $K(\sqrt[3]{17})$ we get the point 
$$\left(-\frac{7}{17}\sqrt[3]{17}\sqrt[3]{3}^2, -\frac{57}{34}\sqrt{-3} - \frac{1}{2}\right)+\left(-\omega^2\sqrt[3]{3}^2,\frac{3}{2}\sqrt{-3}- \frac{1}{2}\right)\in E^3_{min}(L_{(3,p)}).$$ 
Mapping to $E^3$, we get the point 
$$z=\left(-\frac{28}{17}\sqrt[3]{17}\sqrt[3]{3}^2, -\frac{228}{17}\sqrt{-3}\right)+\left(-4\omega^2\sqrt[3]{3}^2,12\sqrt{-3}\right)\in E^3(L_{(3,p)}).$$  
Mapping to $E^1$ further, we get the point
$$\phi(z)=\left(-\frac{28}{17}\sqrt[3]{17}, -\frac{76}{17}\sqrt{-3}\right)+\left(-4\omega^2,4\sqrt{-3})\right)\in E^1(L_{(p)}).$$  
Note that
$$\phi(z)-(-4, 4\sqrt{-3})=\left(\left(-\frac{38}{49}\sqrt{-3} + \frac{34}{49}\right)\sqrt[3]{17},-\frac{1292}{343}\sqrt{-3} - \frac{216}{343}\right)$$
is singular mod $\sqrt{-3}$ which is compatible with Theorem \ref{singularcomponent} since we take a degree 2 trace. Also $\phi(z)$ just gives a nontorsion rational point on $E^{17}$ as predicted by our explicit Gross-Zagier formula, since $L(1, E^{9\cdot17})=0$ in this case. Also the polynomial $D(x)$ does not have solutions in $\BF_{17}$,  satisfying the condition in our main theorem.

\bibliographystyle{alpha}
\bibliography{reference}

\begin{thebibliography}{HSY20}

\bibitem[AL70]{AL1970}
A.O.L. Atkin and J.~Lehner.
\newblock Hecke operators on ${\Gamma}_0({N})$.
\newblock {\em Mathematische Annalen}, 185:134--160, 1970.

\bibitem[BS66]{BS}
B.~J. Birch and N.~M. Stephens.
\newblock The parity of the rank of the {M}ordell-{W}eil group.
\newblock {\em Topology}, 5:295--299, 1966.

\bibitem[Cox89]{Cox89}
D.~A. Cox.
\newblock {\em Primes of the Form $x^2+n y^2$}.
\newblock John Wiley $\&$ Sons Inc., 1989.

\bibitem[CST14]{CST14}
L.~Cai, J.~Shu, and Y.~Tian.
\newblock Explicit {G}ross-{Z}agier and {W}aldspurger formulae.
\newblock {\em Algebra $\&$ Number Theory}, 8(10):2523--2572, 2014.

\bibitem[CST17]{CST17}
L.~Cai, J.~Shu, and Y.~Tian.
\newblock Cube sum problem and an explicit {G}ross-{Z}agier formula.
\newblock {\em Amer. Jour. of Math.}, 139(3):785--816, 2017.

\bibitem[Dic13]{Dickson}
L.~E. Dickson.
\newblock {\em History of the Theory of Numbers: Diophantine analysis},
  volume~2.
\newblock Courier Corporation, 2013.

\bibitem[DV09]{DV1}
S.~Dasgupta and J.~Voight.
\newblock Heegner points and sylvester's conjecture.
\newblock {\em Arithmetic Geometry: Clay Mathematics Institute Summer School,
  Arithmetic Geometry, July 17-August 11, 2006, Mathematisches Institut,
  Georg-August-Universit{\"a}t, G{\"o}ttingen, Germany}, 8:91, 2009.

\bibitem[DV18]{DV17}
S.~Dasgupta and J.~Voight.
\newblock Sylvester's problem and mock heegner points.
\newblock {\em Proc. Amer. Math. Soc.}, 146:3257--3273, 2018.

\bibitem[fM]{math}
GH~from MO.
\newblock {\em
  https://mathoverflow.net/questions/346596/how-to-prove-an-equation-involving-sums-of-kronecker-symbol}.

\bibitem[GZ85]{GZ85}
B.~H. Gross and D.~B. Zagier.
\newblock On singular moduli.
\newblock {\em J. Reine Angew. Math.}, 355:191--220, 1985.

\bibitem[GZ86]{GZ1986}
B.~H. Gross and D.~B. Zagier.
\newblock Heegner points and derivatives of {L}-series.
\newblock {\em Inventiones mathematicae}, 84:225--320, 1986.

\bibitem[HSY19]{HSY19}
Y.~Hu, J.~Shu, and H.~Yin.
\newblock An explicit {G}ross-{Z}agier formula related to the {S}ylvester
  conjecture.
\newblock {\em Trans. Amer. Math. Soc.}, 372(10):6905--6925, 2019.

\bibitem[HSY20]{HSY1}
Y.~Hu, J.~Shu, and H.~Yin.
\newblock Waldspurger's period integral for newforms.
\newblock {\em Acta Arith.}, 195(2):177--197, 2020.

\bibitem[Jed05]{Jed}
T.~Jedrzejak.
\newblock Height estimates on cubic twists of the {F}ermat elliptic curve.
\newblock {\em Bull. Austral. Math. Soc.}, 72(2):177--186, 2005.

\bibitem[KM88]{KM1988}
M.~A. Kenku and F.~Momose.
\newblock Automorphism groups of the modular curves ${X}_0({N})$.
\newblock {\em Compositio Mathematica}, 65(1):51--80, 1988.

\bibitem[Lan87]{Lang-ef}
S.~Lang.
\newblock {\em Elliptic functions}, volume 112 of {\em Graduate Texts in
  Mathematics}.
\newblock Springer, 1987.

\bibitem[Lig70]{Ligozat}
G.~Ligozat.
\newblock Fonction l des courbes modulaires.
\newblock {\em S{\'e}minaire Delange-Pisot-Poitou. Th{\'e}orie des nombres},
  11(1):1--10, 1969-1970.

\bibitem[LV15]{LV}
K.~Lauter and B.~Viray.
\newblock On singular moduli for arbitrary discriminants.
\newblock {\em Int. Math. Res. Not. IMRN}, (19):9206--9250, 2015.

\bibitem[Mor07]{Morain}
F.~Morain.
\newblock Computing the cardinality of {CM} elliptic curves using torsion
  points.
\newblock {\em J. Th\'{e}or. Nombres Bordeaux}, 19(3):663--681, 2007.

\bibitem[Sag]{Sage}
Sagemath.
\newblock {\em https://cloud.sagemath.com/projects}.

\bibitem[Sel51]{Selmer51}
E.~S. Selmer.
\newblock The {D}iophatine equation $ax^3+by^3+cz^3=0$.
\newblock {\em Acta Math.}, 87:203--362, 1951.

\bibitem[Shi94]{Shimurabook}
G.~Shimura.
\newblock {\em Introduction to the arithmetic theory of automorphic functions},
  volume~11 of {\em Publications of the Mathematical Society of Japan}.
\newblock Princeton University Press, Princeton, NJ, 1994.
\newblock Reprint of the 1971 original, Kan{\^o} Memorial Lectures, 1.

\bibitem[Sil92]{silvermanbook1}
J.~H. Silverman.
\newblock {\em The arithmetic of elliptic curves}, volume 106 of {\em Graduate
  Texts in Mathematics}.
\newblock Springer-Verlag, New York, 1992.
\newblock Corrected reprint of the 1986 original.

\bibitem[Sil94]{Silvermanbook2}
J.~H. Silverman.
\newblock {\em Advanced topics in the arithmetic of elliptic curves}, volume
  151 of {\em Graduate Texts in Mathematics}.
\newblock Springer-Verlag, New York, 1994.

\bibitem[SSY]{SSY}
J.~Shu, X.~Song, and H~Yin.
\newblock Cube sums of the forms $3p$ and $3p^2$.
\newblock {\em Mathematische Zeitschrift
  https://doi.org/10.1007/s00209-021-02730-w}.

\bibitem[ST68]{ST}
J-P. Serre and J.~Tate.
\newblock Good reduction of abelian varieties.
\newblock {\em Ann. of Math. (2)}, 88:492--517, 1968.

\bibitem[Ste68]{Stephens}
N.~M. Stephens.
\newblock The diophantine equation {$X^{3}+Y^{3}=DZ^{3}$} and the conjectures
  of {B}irch and {S}winnerton-{D}yer.
\newblock {\em J. Reine Angew. Math.}, 231:121--162, 1968.

\bibitem[Tat75]{Tate75}
J.~Tate.
\newblock Algorithm for determining the type of a singular fiber in an elliptic
  pencil.
\newblock In {\em Modular functions of one variable, {IV} ({P}roc. {I}nternat.
  {S}ummer {S}chool, {U}niv. {A}ntwerp, {A}ntwerp, 1972)}, pages 33--52.
  Lecture Notes in Math., Vol. 476, 1975.

\bibitem[Tia14]{Tian}
Y.~Tian.
\newblock Congruent numbers and {H}eegner points.
\newblock {\em Camb. J. Math.}, 2(1):117--161, 2014.

\bibitem[Wit01]{Wittmann}
Christian Wittmann.
\newblock Group structure of elliptic curves over finite fields.
\newblock {\em J. Number Theory}, 88(2):335--344, 2001.

\bibitem[Zag]{Zagier}
D.~Zagier.
\newblock Modular forms of one variable, notes based on a course given in
  utrecht, spring 1991.
\newblock {\em
  https://people.mpim-bonn.mpg.de/zagier/files/tex/UtrechtLectures/UtBook.pdf}.

\end{thebibliography}
\end{document}